\documentclass{amsart}

\usepackage{tikz-cd}
\usepackage{amssymb}
\usepackage{amsthm}
\usepackage{amsmath}
\usepackage{xfrac}
\usepackage{enumerate}
\usetikzlibrary{decorations.pathmorphing}
\usepackage{hyperref}

\newcommand{\E}{\mathcal{E}}

\DeclareMathOperator{\Id}{Id}
\newcommand{\inv}{^{-1}}
\newcommand{\W}{\mathcal{W}}
\renewcommand{\L}{\mathcal{L}}
\DeclareMathOperator{\Spec}{Spec}

\newcommand{\g}{\mathfrak{g}}

\newcommand{\Mod}{\overrightarrow{\mathbf{Mod}}}
\newcommand{\CMod}{\overleftarrow{\mathbf{Mod}}}
\newcommand{\LR}{\overrightarrow{\mathbf{LRP}}}
\newcommand{\CLR}{\overleftarrow{\mathbf{LRP}}}
\newcommand{\SCn}[1]{\X_{[0,1]^{#1}}}
\newcommand{\SC}{\X_{[0,1]} }

\newcommand{\til}{\widetilde}
\newcommand{\m}{\mathfrak{m}}

\DeclareMathOperator{\Hom}{Hom}
\newcommand{\C}{\mathcal{C}}

\newcommand{\D}{\mathcal{D}}
\newcommand{\X}{\mathfrak{X}}
\newcommand{\B}{\mathcal{B}}
\newcommand{\R}{\mathbb{R}}
\DeclareMathOperator{\pr}{pr}
\renewcommand{\O}{\mathcal{O}}

\newcommand{\BBR}{\mathbb{R}}
\newcommand{\F}{\mathcal{F}}
\newcommand{\grpd}{\rightrightarrows}

\DeclareMathOperator{\Der}{Der}

\DeclareMathOperator{\Fib}{Fib}

\newcommand{\A}{\mathcal{A}}

\newcommand{\Cinf}{\mathcal{C}^\infty}
\DeclareMathOperator{\difer}{d}
\newcommand{\dif}{\difer\!}
\newcommand{\difn}[1]{{\difer}^{#1}\!}
\newcommand{\difnp}[2]{{\difer}^{#1}_{#2}}

\newcommand{\LieRS}{\overrightarrow{\mathbf{LieRS}}}
\newcommand{\CLieRS}{\overleftarrow{\mathbf{LieRS}}}

\makeatletter
\newtheorem*{rep@theorem}{\rep@title}
\newcommand{\newreptheorem}[2]{%
\newenvironment{rep#1}[1]{%
 \def\rep@title{#2 \ref{##1}}%
 \begin{rep@theorem}}%
 {\end{rep@theorem}}}
\makeatother

\newtheorem{theorem}{Theorem}
\newreptheorem{theorem}{Theorem}
\newtheorem{proposition}[theorem]{Proposition}
\newtheorem{lemma}[theorem]{Lemma}

\theoremstyle{definition}
\newtheorem{definition}[theorem]{Definition}

\theoremstyle{remark}
\newtheorem{remark}[theorem]{Remark}
\newtheorem{example}[theorem]{Example}

\numberwithin{theorem}{section}

\title{On sheaves of Lie-Rinehart algebras}
\author{Joel Villatoro}
\address{Department of Mathematics\\
  KU Leuven\\
  3001, Leuven BE}
\email{joel.villatoro@kuleuven.be}

\begin{document}

\begin{abstract}
We study sheaves of Lie-Rinehart algebras over locally ringed spaces. We introduce morphisms and comorphisms of such sheaves and prove factorization theorems for each kind of morphism. Using this notion of morphism, we obtain (higher) homotopy groups and groupoids for such objects which directly generalize the homotopy groups and Weinstein groupoids of Lie algebroids. We consider, the special case of sheaves of Lie-Rinehart algebras over smooth manifolds. We show that, under some reasonable assumptions, such sheaves induce a partition of the underlying manifold into leaves and that these leaves are precisely the orbits of the fundamental groupoid.
\end{abstract}
\maketitle
\begin{center}
Department of Mathematics\\
Celestijnenlaan 
KU Leuven\\
3001, Leuven BE
\end{center}
\tableofcontents
\section{Introduction}\label{section:introduction}
Lie algebroids are an important class of geometric structures which are commonly used to study certain infinite dimensional Lie algebras. 
However, these tools are limited to infinite dimensional Lie algebras which arise as the vector space of sections of a vector bundle. 
However, many infinite dimensional Lie algebras are not of this kind. 
Despite this, many such algebras still exhibit algebroid-like behavior.
For example, they may be the sections of a sheaf of modules rather than a genuine vector bundle. 
The most notable examples of this behavior are submodules of a Lie algebroid (singular subalgebroids) of which submodules of vector fields (singular foliations) form a well studied class.

Our goal is to study these general algebroid-like objects using the theory of Lie-Rinehart algebras, also known as Lie pseudoalgebras. 
These algebras first appeared in the work of Rinehart\cite{LR_Original}. 
They were studied in more detail by Huebschmann\cite{Hueb} who coined the term Lie-Rinehart. 

Lie-Rinehart algebras are an algebraic axiomatization of the basic properties of Lie algebroids. 
Their algebraic nature means that the general theory of Lie-Rinehart algebras do not provide us with many tools to investigate the more geometric elements of Lie Algebroid theory. 
For example, treating a singular foliation as a Lie-Rinehart algebra hides the geometry of the underlying partition of the manifold into leaves. 
Our solution to this problem is to understand singular foliations, singular algebroids, and other similar objects as \emph{sheaves} of Lie-Rinehart algebras which we call \emph{Lie-Rinehart structures}. 
We will see that many of the geometric elements of Lie Algebroid theory (e.g. foliations and fundamental groups) still make sense in this more general setting.

The main motivations for our study of Lie-Rinehart structures come from the following natural examples:
\begin{itemize}
\item Singular subalgebroids: By singular subalgebroid we mean involutive submodules of the sections of a Lie algebroid. In particular, singular foliations are of interest since they appear naturally in many contexts such as the orbits of a group action or as the constraints of motion in control theory.
\item Direct/projective limits of Lie algebroids: Examples of this include the infinite jet bundle of a Lie algebroid. The infinite jet bundle is an `algebroid' which arises as the projective limit of a sequence of Lie algebroid.
\item Fiber products and quotients: Certain desirable operations often fail in the context of Lie algebroids. For example, pullback algebroids, fiber products, and other operations involving Lie algebroids frequently fail to be smooth. However, they are usually perfectly well-defined sheaves of Lie-Rinehart algebras.
\end{itemize}
In this paper, we will explore the general theory of Lie-Rinehart structures which captures all of these examples under a unifying framework. We will first consider Lie-Rinehart structures on locally ringed spaces before proceeding to Lie-Rinehart structures on smooth manifolds. 

Although this paper formally works with locally ringed spaces over \( \R \), an argument can be made that a more natural setting would be locally \( \Cinf \)-ringed spaces which are often used in synthetic differential geometry. Indeed, it seems that essentially all of the theorems in this document can be directly transported to that setting with almost no modification. Ultimately, since there is not yet a clear advantage for using \( \Cinf \)-rings in this context we have elected to use the less technical of the two formalisms.

\subsection*{Main Results}
The first of our main theorems concerns the general case of Lie-Rinehart structures on a locally ringed space:
\begin{theorem}\label{theorem:weinstein.groupoid}
Suppose \( \A \) is a sheaf of Lie-Rinehart algebras over a locally ringed space \(X \) over \( \R \). There exists a functorial construction of a groupoid \(\Pi(\A) \grpd X \) over \( X \) which extends the construction of the fundamental/Weinstein groupoid of a Lie Algebroid. That is, when \( X \) is a manifold and \( \A \) is the sheaf of sections of a Lie algebroid then \( \Pi(\A) \) is canonically isomorphic to the fundamental groupoid of the underlying algebroid.
\end{theorem}
After the explanation of our construction, the proof of Theorem~\ref{theorem:weinstein.groupoid} will essentially be a direct corollary of Theorem~\ref{theorem:lie.rinehart.groups.well.defined} where we show that the fundamental groupoid is well-defined.

This theorem suggests that we have successfully constructed the notion of an `integration' of a sheaf of Lie-Rinehart algebras. However, there are some limitations to our integration procedure that we should mention. For one, the groupoid we construct is purely set-theoretic (for now) and at the time of the writing of this document, there is no known differentiation functor to accompany this notion of integration. Work by Androulidakis and Zambon~\cite{AMsheaf} is a partial result in that direction. Also notable is a more algebraic way of constructing an integration of a Lie-Rinehart algebra which appears in another preprint by Ardizzoni, Kaoutit, and Saracco~\cite{ardizzoni2020differentiation}.

In addition to the fundamental groupoid, we are also able to define homotopy groups for Lie-Rinehart structures. Just like the fundamental groupoid, these homotopy groups generalize the notion of the algebroid homotopy groups in the case that the sheaf is the sections of a vector bundle. In the case of algebroids, these homotopy groups are known to control several important invariants such as the obstruction to integrability~\cite{CrFeLie}. An interesting point of comparison arises in a preprint by Laurent-Gengoux, Lavau and Strobl\cite{SylvainArticle} where they define higher homotopy groups for a reasonably general class of singular foliations. Since singular foliations are a special case of Lie-Rinehart structures, it is reasonable to wonder whether the invariants constructed in this paper coincide with their work. At this time, the author does not know the answer to this question.

For manifolds, we can identify a class of particularly nice Lie-Rinehart structures which we call \emph{adjoint integrable}. Essentially, an adjoint integrable sheaf is one for which the adjoint representation of the associated Lie algebra can be exponentiated to a time-dependent family of automorphisms. Notably, any (locally) finitely generated involutive submodule of a Lie algebroid is adjoint integrable as well as many non-finitely generated examples.

Our second main result regards adjoint integrable Lie-Rinehart algebras:
\begin{theorem}\label{theorem:partition.into.leaves}
Suppose \( \A \) is a adjoint integrable sheaf of Lie-Rinehart algebras over a manifold \( M \). Then \( \A \) induces a singular foliation on \( M \) and the isomorphism classes of \( \Pi(\A) \grpd M \) are precisely the associated partition of \( M \) into leaves.
\end{theorem}
Theorem~\ref{theorem:partition.into.leaves} certainly suggests that adjoint integrable Lie-Rinehart algebras indeed seem to behave very similarly to classical algebroids. This similarity  merits further investigation beyond the scope of this paper. There are plenty of examples of adjoint integrable Lie-Rinehart structures. For instance, any locally finitely generated singular subalgebroid is automatically adjoint integrable. There are many infinitely generated examples as well.
\subsection*{Outline}
In Section~\ref{section:lie.rinehart.algebras} we will give an exposition of some of the basic elements of the theory of Lie-Rinehart algebras also known as Lie-Rinehart pairs or Lie pseudoalgebras. In \cite{hm1993}, Higgins and Mackenzie noted the existence of two natural notions of morphism between Lie-Rinehart algebras which they called morphisms and comorphisms. They also proved a factorization theorem for Lie-Rinehart morphisms. In Section~\ref{section:lie.rinehart.algebras} we will complete the picture with Theorem~\ref{theorem:lie.rinehart.pair.comorphism.factors} which is a factorization result for Lie-Rinehart morphisms.

Section~\ref{section:ringed.spaces}, we will discuss the basics for locally ringed spaces and sheaves of modules over locally ringed spaces. We will assume that the reader is mostly familiar with these concepts. A much more detailed discussion of these topics can be found in most standard texts on algebraic geometry such as ~\cite{hartshorne} or \cite{stacks-project}. We will also address the topic of \emph{base-changing} morphisms and comorphisms of sheaves of modules. That is, morphisms of sheaves of modules where the underlying space is allowed to change.

Section~\ref{section:fiber.determined.modules}, is dedicated to looking at a special case of sheaves of modules which we call \emph{fiber determined}. In short, a fiber determined sheaf of modules is a sheaf of modules for which the sections are completely determined by their restriction to each point. The main result of this section is Theorem~\ref{lemma:geometric.module.has.calculus} which says, roughly, that it is possible to differentiate and integrate time-dependent sections of fiber determined modules over manifolds.

Section~\ref{section:lie.rinehart.structures} discusses sheaves of Lie-Rinehart algebras over locally ringed spaces. We call such things Lie-Rinehart structures. A Lie-Rinehart structure directly generalizes the notion of a Lie algebroid. When the underlying sheaf of modules is the sheaf of sections of a vector bundle, these two notions coincide. We define a notion of both morphism and comorphisms of Lie-Rinehart structures. In Theorems~\ref{theorem:lie.rinehart.structure.morphism.factors} and \ref{theorem:lie.rinehart.pair.comorphism.factors} we see that the factorization of Lie-Rinehart algebra morphisms and comorphisms give rise to factorization systems at the level of sheaves. 

Section~\ref{section:frobenius.integrability} addresses a particularly nice class of Lie-Rinehart structures on smooth manifolds which we call \emph{adjoint integrable}. In short, adjoint integrable means that the adjoint representation of the Lie algebra of sections is `integrable' to a time-dependent family of automorphisms. We will see several examples of adjoint integrable Lie-Rinehart structures. With Theorem~\ref{theorem:frobenius.int.adjoint.flow}, we learn that sections of a adjoint integrable Lie-Rinehart structure give rise to automorphisms of the sheaf. Next is Theorem~\ref{thm:frob.int.foliation} proves the existence of a foliation on the underlying manifold. Finally, Theorem~\ref{thm:fin.gen.frob.int} shows that any fiber-determined and locally finitely generated Lie-Rinehart structure on a smooth manifold is automatically adjoint integrable. This last result is particularly relevant to the study of singular foliations where finitely generated subsheaves of vector fields are a frequent topic of study.

Section~\ref{section:lie.rinehart.homotopy} provides us with the basic definitions for the homotopy theory of Lie-Rinehart structures. In the case where the Lie-Rinehart structure is the sheaf of sections of a Lie algebroid, this homotopy theory is the same thing as algebroid homotopy or \( \A\)-homotopy. There will be a series of examples where we see explain how the usual elements of homotopy theory such as constant homotopies and concatenation of homotopies can be defined in this setting.

The last section, Section~\ref{section:homotopy.groups} is dedicated to constructing the (higher) homotopy groups and fundamental groupoids of a Lie-Rinehart structure. The main result is the proof the group(oid) operation is well-defined. We will also see Theorem~\ref{theorem:weinstein.groupoid} and Theorem~\ref{theorem:partition.into.leaves} arise as immediate corollaries of our construction.

\subsection*{Acknowledgments}
The author would like to thank Marco Zambon for the numerous discussions which on this topic over the years. He also appreciates the contributions of Alfonso Garmendia who was a coauthor on \cite{villatoro2019integration} which served as a starting point for this project. Finally, the author is indebted to his former PhD advisor Rui Fernandes who provided much useful feedback. This work was supported by funding from Fonds Wetenschappelijk Onderzoek (FWO Project G083118N).

\subsection*{Notation}
\begin{itemize}
\item Letters \( X \) and \( Y \) will be used to denote ringed spaces with structure sheaves \( \O_X \) and \( \O_Y \).
\item Given a map of ringed spaces \( f \colon X \to Y \), we write \( f^\# \) to denote the associated ring homomorphism.
\item \( M \) and \( N \) will be used for smooth manifolds. 
\item Given a \emph{smooth} function \( f \colon M \to N \) we will write \( f^\# \) to denote the pullback homomorphism for smooth functions.
\item The symbols \( \E \) and \( \W \) will denote sheaves of modules.
\item The symbols \( \A \) and \( \B \) will denote sheaves of Lie-Rinehart algebras.
\item Given a fiber bundle \( \pi \colon E \to X \) we write \( \Gamma_E \) to denote the sheaf of sections.
\end{itemize}
\section{Lie-Rinehart Algebras}\label{section:lie.rinehart.algebras}
The notion of a Lie-Rinehart algebra can be obtained by starting with a Lie algebroid and then retaining only the associated algebraic structure. 
The origin of the notion appears originally in a paper of Rinehart~\cite{Rinehart} and Huebschmann\cite{Hueb} gave Lie-Rinehart algebras their name. 
In this section we will review the purely algebraic theory of Lie-Rinehart algebras.
\subsection{Modules}
We will begin by a discussion of morphisms and comorphisms of modules of \( \R \)-algebras. 
Throughout this text, we will assume all algebras are commutative, unital algebras over the real numbers. 

We need to allow our morphisms to be well defined between modules with different underlying rings. 
To that end, we will follow the approach of Higgins and Mackenzie\cite{hm1993}. 
Higgins and Mackenzie noted that there are actually two natural notions of base changing morphisms of modules. 
These two notions are, in a sense, dual to one another.
\begin{definition}\label{defn:module.pair}
A \textbf{module pair} is a pair  \( (M, R) \) where \( R \) is an \( \R \)-algebra and \( M \) is an \( R \)-module. A \textbf{morphism} of module pairs:
\[ (F, \phi ) \colon (M,R) \to (N,S) \] 
consists of a morphism of algebras \( \phi \colon R \to S \) and a morphism of \( R\)-modules \( F \colon M \to N \).

A \textbf{comorphism} of module pairs:
\[ (F, \psi) \colon (M,R) \to (N,S) \] 
consists of a homomorphism \( \psi \colon S \to R  \) and a morphism of \( R \)-modules \( F \colon M \to R \otimes_S N \). 
\end{definition}
We will write \( \Mod \) and \( \CMod \) to denote the categories of module pairs with morphisms and comorphisms as arrows, respectively. 
There is a natural duality between these categories. 
Given a morphism \( (F,\phi) \colon (M,R) \to (N,S) \) in \( \Mod \) one can construct a comorphism: 
\[ (F^*, \phi) \colon ( \Hom_S(N,S) , S) \to (\Hom(M,R) , R)   \]
This construction is actually contravariant functor \(\Mod \to \CMod \). 
If we restrict to finitely generated projective modules, this contravariant functor is an equivalence of categories.

Let us see some examples of morphisms and comorphisms in a geometric context.
\begin{example}\label{example:module.pair.vector.bundle}
Suppose \( E \to X \) is a vector bundle over a smooth manifold \( X \). 
Let \( U \) be an open subset of \( X \). The natural restriction homomorphism: 
\[  (\Gamma_E(X), C^\infty_X(X)) \to (\Gamma_E(U), \Cinf_X(U))  \]
yields a morphism of module pairs.
\end{example}
\begin{example}\label{example:module.pair.comorphism.vb.morphism}
Suppose \( E \to X \) and \( W \to Y \) are a vector bundles and \( T \colon E \to W \) is a vector bundle morphism covering \( f \colon X \to Y \). Write \( f^* W := W \times_X Y \) to denote the pullback bundle.
There is a natural identification: 
\[ \Gamma_{f^* W} (X) \cong C^\infty_X(X) \otimes_{C^\infty_Y(Y)} \Gamma_W( Y )  \] 
and any vector bundle morphism has a pushforward operation: 
\[ T_* \colon \Gamma(E) \to \Gamma( f^* W)  \] 
Therefore:
\[ (T_*, f^\#) \colon (\Gamma(E), C^\infty(X) ) \colon ( \Gamma(W) , C^\infty(Y) ) \]
is a comorphism of module pairs.  
\end{example}
\subsection{Lie-Rinehart pairs}
Briefly, a Lie-Rinehart pair is a module pair \( (\A, R) \) which has been equipped with a Lie algebra structure on \( \A \) and an \( \A \)-module structure (as a Lie algebra) on \( R \) which is compatible with the \( R \)-module structure on \( \A \) via a Leibnitz identity. The model example one should keep in mind for Lie-Rinehart pairs is the pair consisting of vector fields on a manifold together with smooth functions \( (\mathfrak{X}_M(M), \Cinf_M(M)) \).
\begin{definition}\label{definition:lie.rinehart.pair}
A \textbf{Lie-Rinehart(LR) pair} is a module pair \( (\A, R) \) with the following additional data:
\begin{itemize}
\item a Lie bracket on \( \A \):
\[ \A \times \A \to \A  \qquad  (a, b) \mapsto [a, b]  \]
\item a Lie algebra homomorphism called the \textbf{anchor map}: 
\[ \rho_\A \colon \A \to \Der(R) \qquad a \mapsto \L_a \]
\end{itemize}
such that for all \( v, w \in \A \) and \( r \in R \):
\[ [ v, rw]  = r [v, w] + \L_a(r) w \]
\end{definition}
\begin{example}\label{example:lie.rinehart.pair.algebroid}
Let \( \mathcal{A} \to M \) be a Lie algebroid with anchor map \( \rho \colon \mathcal{A} \to TM \). Then \( (\Gamma_\A(M), C^\infty(M)) \) is an LR pair.
\end{example}
\begin{example}\label{example:lie.rinehart.pair.derivations}
Suppose \( R \) is a commutative unital algebra. Then \( (\Der(R) , R) \) is canonically a LR pair.
\end{example}
One can also construct Lie-Rinehart structures out of a differential on an exterior algebra. 
To the authors knowledge, this correspondence was originally observed by Vaintrob\cite{Vaintrob} in the context of exterior algebras of vector bundles.
\begin{example}\label{example:lie.rinehart.pair.dga.dual}
Suppose \( R \) is a commutative unital algebra and let \( E \) be an \( R \)-module. Suppose we are given a differential \( \difer \colon \wedge^\cdot E \to \wedge^\cdot E \) on the exterior algebra of \( E \). In other words, \( \difer \) is a degree 1 graded derivation which squares to zero. Let \( \A := \Hom(E,R) \). Then:
\[ \rho \colon \A \to \Der(R) \qquad \L_a r := \langle a , \dif r \rangle \]
We can also define a Lie bracket \( A \times \A \to \A \):
\[\langle [a_1, a_2 ], e \rangle := \L_{a_1} \langle a_2 ,e \rangle - \L_{a_2} \langle a_1, e \rangle - \langle a_1 \wedge a_2 , \dif e \rangle  \]
A standard calculation shows that these operations satisfy the Leibniz rule and Jacobi identity. Therefore \( (\A , R) \) is an LR pair.
\end{example}
Recall that there are two notions of morphism for module pairs. We should not be surprised to learn that there are two corresponding notions of morphism for Lie-Rinehart pairs. Each one is useful in the appropriate context and come with their own tradeoffs.
\begin{definition}\label{definition:lie.rinehart.pair.morphism}
Suppose \( ( \A, R) \) and \( (\B, S) \) are LR pairs. A \textbf{Lie-Rinehart(LR) morphism} \( (F, \phi ) \colon ( \A, R) \to (\B ,S) \) is a morphism of module pairs such that:
\begin{enumerate}
\item (compatibility with the anchor) for all \( a \in \A \) and \( r \in R \):
\[  \phi (\L_a(r)) = \L_{F(a)} \phi(r).  \]
Equivalently, the following diagram commutes:
\[ \begin{tikzcd}
\A \arrow[d] \arrow[r, "F"]&   \B \arrow[d] \\
\Der(R) \arrow[r]& \Der(R,S)  
\end{tikzcd}  \]
\item (compatibility with the bracket) \( F \colon \A \to \B \) is a Lie algebra homomorphism.
\end{enumerate}
The category of LR pairs with LR morphisms is denoted \( \LR \)
\end{definition}
This is the simplest notion of morphism that exists between LR pairs. However, it does not correspond to the usual notion of morphism of Lie algebroids. Indeed, recall that we saw earlier that morphisms of module pairs do not generally give rise to morphisms of vector bundles.
\begin{example}\label{example:lie.rinehart.pair.morphism.algebroid.restriction}
Suppose \( X \) is a manifold and \( A \to X \) is a Lie algebroid. Given an open subset \( U \subset X \) the restriction morphism \( (\Gamma_A(X) , \Cinf_X(X)) \to (\Gamma_A(U), \Cinf_X(U) ) \) is a LR morphism.
\end{example}
The correct algebroid analogue to a morphism of LR pairs is called a \emph{comorphism} of Lie algebroids\cite{hm1993}. Comorphisms appear mostly in the context of Poisson geometry and Lie algebroid actions.
\begin{example}\label{example:lie.rinehart.pair.morphism.algebroid.comorphism}
Suppose \( A \to X \) and \( B \to Y \) are Lie algebroids. Suppose \( (\Phi, f) \colon A \to B  \) is a Lie algebroid comorphism. This means that \( f \colon Y \to X \) is a smooth map and \( \Phi \colon f^* A \to B \) is a vector bundle map such that:
\[ \Phi^\dag \colon  \Gamma_A(X) \to \Gamma_B(Y)  \qquad \alpha \mapsto \Phi(f^*\alpha) \]
is a Lie algebra homomorphism and for all \( y \in Y \) and \( a \in A_{f(y)} \):
\[  \difnp{}{y} f \circ \rho_B \circ \Phi(y,a)  =  \rho_A(a)    \]

In this case, the pair \( (\Phi^\dag , f^\sharp) \colon  (\Gamma_A(X),\Cinf_X(X)) \to (\Gamma_B(Y) , \Cinf_Y(Y)) \) is a morphism of LR pairs. 
\end{example}
\begin{example}\label{example:lie.rinehart.pair.morphism.poisson.map}
Suppose \( (M, \pi_M) \) and \( (N, \pi_N ) \) are Poisson manifolds and \( f \colon M \to N \) is a Poisson map. The associated map \(F \colon \Omega^1(M) \to \Omega^1(N) \) is a Lie algebra homomorphism and the pair \( (F, f^\#) \) is a morphism of Lie-Rinehart algebras. 
\end{example}
In order to study ordinary Lie algebroid morphisms, we must turn to LR comorphisms. The definition of a LR comorphism is somewhat more complicated than the one for morphisms. This is mainly due to the extra complications with defining compatibility with the bracket.
\begin{definition}\label{definition:lie.rinehart.pair.comorphism}
Suppose \( ( \A, R) \) and \( (\B, S ) \) are Lie-Rinehart pairs. A \textbf{comorphism} \( (F, \psi ) \colon (\A, R) \to (\B, S) \) consists of a comorphism of modules such that:
\begin{enumerate}[(a)]
\item (Compatibility with Anchor) for all \( s \in S \) and \( a \in \A \): 
\[ F(a) = \sum_{i=1}^n r^i \otimes s_i  \quad \Rightarrow \quad  \L_a( \psi(s)) = \sum_{i=1}^n r^i \psi(\L_{b_i}s) \]
Equivalently, the following diagram commutes:
\[
\begin{tikzcd}
\A \arrow[r, "F"] \arrow[d]  &  R \otimes_S \B \arrow[d]\\
\Der(R) \arrow[r] & \Der(S,R)  
\end{tikzcd}
\]
\item (Compatibility with Lie brackets) For all \( a, b \in \A \) such that
\begin{equation}\label{eqn:comorphism.parameterization}  F(a) = \sum_{i=1}^n c^i \otimes e_i \qquad F(b) = \sum_{i=1}^m d^j \otimes f_j 
\end{equation}
we have that:
\[ F([a, b]) = \sum_{i=1}^n \sum_{j=1}^m c^i d^j \otimes [e_i, f_j] + \sum_{j=1}^m \L_a (d^j ) \otimes f_j - \sum_{i=1}^n \L_b(c^i) \otimes e_i \]
\end{enumerate}
The category of LR pairs with LR comorphisms is denoted \( \CLR \).
\end{definition}
\begin{example}\label{example:lie.rinehart.pair.comorphism.algebroid.morphism}
Suppose \( A \to M \) and \( B \to  N \) are Lie algebroids with \( M \) compact. Then any Lie algebroid morphism \( \Phi \colon A \to B \) gives rise to a comorphism \( (\Phi_*, f^\#) \) of the Lie-Rinehart algebras.
\end{example}
\subsection{Factorization of (co)morphisms}
LR morphisms and comorphisms both have nice factorization properties. The factorization operation for LR morphisms was originally observed in \cite{hm1993} and is closely related to actions of Lie algebroids. One can also find further discussion on the topic in Chen and Liu~\cite{ChenLiu}.
\begin{definition}\label{definition:lie.rinehart.pair.action}
Suppose \( (\A, R) \) is an LR pair and \( \phi \colon R \to S \) is a homomorphism of algebras. An \textbf{action} of \( (\A, R) \) along \( \phi \) consists of a \( R \)-linear Lie algebra homomorphism \( \alpha \colon \A \to \Der(S) \).
\end{definition}
\begin{example}\label{example:lie.rinehart.pair.action.poisson.map}
Suppose \( (M, \pi_M) \) and \( (N, \pi_N) \) are a Poisson manifolds and \( f \colon N \to M \) is a Poisson map. By pulling back differential forms we get a \( \Cinf_M(M) \)-module homomorphism \( \Omega^1(M) \to \Omega^1(N) \). If we compose this with \( \pi_N^\sharp \colon \Omega^1(N) \to \X_N(N) \), we obtain an action of \( (\Omega^1(M), \Cinf_M(M)) \) on \( \Cinf_N(N) \) along \( f^\# \). 
\end{example}
\begin{example}\label{example:lie.rinehart.pair.action.lie.rinehart.morphism}
Suppose \( (F, \phi) \colon (\A, R) \to (\B, S) \) is an LR morphism. Then \(  \rho_\B \circ F \) defines an action of \( (\A, R) \) along \( \phi \).
\end{example}
One can generalize the preceding example to actions of Lie algebroids.
\begin{lemma}\label{lemma:lie.rinehart.pair.action.defines.morphism}
Suppose \( (\A , R) \) is an LR pair and \( \phi \colon R \to S \) is a homomorphism of algebras. If \( \alpha \colon \A \to \Der(S) \) is an action along \( \phi \) then there is a unique LR structure on \( (S \otimes_R \A , S) \) which satisfies the following properties:
\begin{itemize}
\item For all \( s_1, s_2 \in S \) and \( a \in \A \): 
\[  \L_{s_1 \otimes a} s_2 = s_1 (\alpha(a)s_2)  \]
\item The morphism of module pairs:
\[  (1 \otimes \Id_\A , \phi ) \colon (\A, R) \to (S \otimes_R \A, S)  \]
is an LR morphism.
\end{itemize}
\end{lemma}
For a proof of this lemma, we refer the reader to Proposition 5.11 in \cite{hm1993}.
\begin{theorem}[Higgins-Mackenzie]\label{theorem:lie.rinehart.pair.morphism.factors}
Suppose \( (F,\phi) \colon (\A, R) \to (\B, S) \) is an LR morphism. Let us equip \( (S \otimes_R \A , S) \) with the LR structure induced by the action \( \rho_\B \circ F \). Then there exists a unique LR morphism 
\[ (\overline{F}, \Id_R ) \colon (S \otimes_R \A , S) \to (\B, S)  \]
such that 
\[ (F, \phi) = (\overline{F}, \Id_R) \circ (1 \otimes \Id_A , \phi) \]
\end{theorem}
\begin{proof}
Let us define \( \overline{F} \colon S \otimes_R \A \to B \) on basic tensors by the equation:
\[ \overline{F}(s \otimes a) := s F(a) \]
From this definition, it is clear that this is the unique such \( S \)-module homomorphism which satisfies \( \overline{F}(1 \otimes \Id_\A) = F \).
\end{proof}
In words, this theorem says that any LR morphism can be factored into a morphism which is the identity map at the level of rings and an LR morphism coming from an action along a ring homomorphism.

There is also a factorization theorem for LR comorphisms. In order to state this result, we need the notion of pushing an LR pair along a ring homomorphism.
\begin{definition}\label{definition:lie.rinehart.pair.base.change}
Suppose \( (\B, S) \) is an LR pair and \( \psi \colon S \to R \) is a ring homomorphism. The \textbf{base change} \(( \psi^!\B, R) \) of \( \B \) along \( \psi \) is defined as follows.

As a module, it is the fiber product of \( \Der(S) \) with \( R \otimes_S \B \) over \( \Der(S,R) \). In other words, it makes the following diagram a pullback square of \( R \)-modules:
\begin{equation}\label{eqn:base.change.diagram}
\begin{tikzcd}
\psi^! \B \arrow[d] \arrow[r] & R \otimes_S \B \arrow[d] \\
\Der(R) \arrow[r] & \Der(S,R) 
\end{tikzcd}
\end{equation}
By construction, there is a canonical \( R \)-module map \( \psi_* \B \to \Der(R) \) which we take to be the anchor map. Finally, given two elements of \(R \otimes_S \B \):
\[ \left(D , \sum_{i}^n c^i \otimes e_i \right) \quad \text{ and } \quad \left(D' , \sum_{j}^m d^i \otimes f_j \right)  \]
we define the Lie bracket to be:
\[ \left( D D' - D' D , \  \sum_{i=1}^n \sum_{j=1}^m c^i d^j \otimes [e_i, f_j] + \sum_{j=1}^m D(d^j ) \otimes f_j - \sum_{i=1}^n D'(c^i) \otimes e_i \right) \]
\end{definition}
Checking that this results in a well-defined LR pair is straightforward but rather tedious so we will leave that task to the skeptical reader. An important feature of this construction is that there is a canonical projection from the base change to the original LR pair:
\[ (\pi , \psi) \colon (\psi^! \B , R) \to (\B, S)  \]
\begin{example}\label{example:lie.rinehart.pair.base.change.algebroid}
Suppose \( A \to X \) is a Lie algebroid and \( f \colon P \to X \) is a submersion with \( P \) compact. Then, as a vector bundle, the pullback algebroid \( f^! A \) is defined to be the fiber product:
\[  f^! A :=  T M \times_{\dif f, \rho_A} A \]
It is a standard check to show that this canonically results in a Lie algebroid. 
Furthermore, there is a natural isomorphism:
\[  (\Gamma_{f^! A} (P), \Cinf_P(P)) \cong ((f^\#)^! \Gamma_\A(M), \Cinf_P(P) )  \]
\end{example}
So, we see that base change operation is a little tricky to define, but at the geometric level, it corresponds to a very common Lie algebroid operation. We can use the base change to state our factorization theorem for LR comorphisms.
\begin{theorem}\label{theorem:lie.rinehart.pair.comorphism.factors}
Suppose \( (F, \psi) \colon (\A, R) \to (\B, S) \) is an LR comorphism. Then there exists a unique LR comorphism \( (\underline{F}, \Id_R) \colon (\A, R) \to (\psi^! \B, R) \) such that:
\[ (F, \psi) =  (\pi_2, \psi) \circ (\underline{F}, \Id_R) \]
\end{theorem}
\begin{proof}
For \( a \in \A \), let:
\[ \underline{F} (a) := (\L_a , F(a) ) \]
We claim that \( (\underline{F} ,\Id_R) \colon (\A, R) \to (\psi^! \B, R) \) is an LR comorphism. The compatibility with the anchor map follows from the commutativity of (\ref{eqn:base.change.diagram}). Compatibility with the bracket follows immediately from substituting the relevant terms.
\end{proof}
\section{Ringed Spaces and sheaves of modules}\label{section:ringed.spaces}
We will now leave the world of pure algebra and introduce more geometry into the discussion. The most general geometric concept that we will use in this paper is that of a ringed space.
\subsection{Standard sheaf constructions}
Although we assume the reader is familiar with the notion of a sheaf of rings, we will spend some time establishing our notational conventions for some standard sheaf constructions to avoid confusion.
\begin{definition}[Image]\label{definition:image.sheaf}
Suppose \( \E \) is a sheaf on a topological space \( X \) and let \( f \colon X \to Y \) be a continuous map. The \textbf{image} \(f_* \E \) of \( \E \) along \( f \) is defined to be the following sheaf on \( Y \):
\[ U \subset Y \, \text{ open} \qquad U  \mapsto  \E(f^{-1}(U) ) \]
\end{definition}
The image sheaf is probably the easiest way to move a sheaf from one topological space to another. It is slightly less easy, but often necessary, to move a sheaf in the reverse direction.
\begin{definition}[Inverse image]\label{definition:inverse.image.sheaf}
Suppose \( \F \) is a sheaf on a topological space \( Y \) and let \( f \colon X \to Y \) be a continuous map. The \textbf{inverse image} \(f\inv \E \) of \( \E \) along \( f \) is defined to be the sheafification of the following pre-sheaf on \( Y \):
\begin{equation*} 
U \subset Y \, \text{open } \qquad  U \mapsto \left( \mathop{\longrightarrow}^{\lim}_{ V \supset f(U)} \F(V) \right)  
\end{equation*}
\end{definition}
The image and inverse image functors on sheaves are related by a well-known adjunction. This adjunction comes from a natural identification:
\begin{equation}\label{eqn:image.adjunction} \Hom_X (f\inv \F, \E) \cong \Hom_Y( \F, f_* \E ) 
\end{equation}
\subsection{Ringed spaces}
A ringed space is a topological space which is equipped with an abstract ring of functions that plays the role of the ring of continuous or smooth functions in other contexts.
\begin{definition}\label{definition:ringed.space}
A \textbf{locally ringed space over \( \R \)} is pair \( (X , \O_X) \) where 
\begin{itemize}
\item \( X \) is a topological space
\item \( \O_X  \) is a sheaf of unital \(\R \)-algebras 
\item For all \( x \in X \) the stalk \( \O_{X, x} \) has a unique maximal ideal (local ring)
\end{itemize}

A \textbf{morphism} of ringed spaces \( (f, f^\#) \colon (X, \O_X) \to (Y, \O_Y) \) consists of a continuous map \( f \colon X \to Y \) and a morphism of sheaves of rings over \(X \)\footnote{ Some authors take \( f^\# \) to be a morphism of sheaves of rings \( \O_Y \to f_* \O_X \). This definition is equivalent as a consequence of Equation~\ref{eqn:image.adjunction}}:
\[ f^\# \colon f\inv \O_X \to \O_Y \]
Finally, for all \( x \in X \) we require that the image of the maximal ideal under the associated map at the level of stalks \( f^\# \colon \O_{Y, f(x)} \to \O_{X, x} \) is contained in the maximal ideal.
\end{definition}
We will usually write \( X \) to denote a ringed space and leave \( \O_X \) implicit. Similarly, we may write \( f \colon X \to Y \) to denote a morphism of ringed spaces and leave \( f^\# \) implicit.

For completeness, let us state the standard examples of locally ringed spaces.
\begin{example}
Suppose \( X \) is a topological space and let \( \C^0_X \) denote the sheaf of continuous functions valued in \( \R \). Then \( (X, \C^0_X ) \) is a locally ringed space.
\end{example}
\begin{example}
Suppose \( M \) is a manifold and let \( \Cinf_M \) denote the sheaf of smooth functions valued in \( \R \), then \( (M, \Cinf_M) \) is a locally ringed space.
\end{example}
\begin{example}
Suppose \( R \) is an \( \R \)-algebra. Let \( \Spec(R) \) denote the set of all prime ideals of \( R \). Let us define the topology on \( \Spec(R) \) by exhibiting its basis. Given \( f \in \Spec(R) \) we define \( U_f \) to be the set of prime ideals which do not contain \( f \). 

Let us define a sheaf of \( \R \)-algebras on \( \Spec(R) \) by using this basis. We define:
\[ \O_{\Spec(R)}(U_f) := R/ \langle f \rangle \]
Where \( \langle f \rangle \) denotes the principal ideal generated by \( f \). This makes \( \Spec(R) \) into a locally ringed space over \( \R \).
\end{example}
\subsection{Sheaves of modules}
\begin{definition}
Suppose \( X \) is a locally ringed space over \( \R \). A sheaf of \( \O_X \)-modules on \( X \) is a sheaf \( \E \) on \( X \) such that the pair \( (\E,\O_X) \) is a sheaf valued in \( \Mod \). In other words, for each open \( U \subset X \), we have that \(( \E(U), \O_X(U)) \) is a module pair and the restriction maps are module pair morphisms.
\end{definition}
We take the point of view that a sheaf of modules is a generalization of the notion of a vector bundle. Indeed, some authors define a vector bundle on a locally ringed space to be a locally free sheaf of modules.
\begin{example}
Suppose \( E \to M \) is a vector bundle over a smooth manifold. Then \( \Gamma_E \) is a sheaf of \( \Cinf_M \)-modules on \( M \).
\end{example}
\begin{example}
Suppose \( X \) is a locally ringed space and \( n \ge 1 \) is an integer. The (standard) free sheaf of rank \( n \) over \( X \) refers the direct sum sheaf:
\[ (\Cinf_X)^{\oplus n}(U) := \Cinf_X(U)^{\oplus n} \]
A sheaf is (locally) free if it is (locally) isomorphic to a standard free sheaf.
\end{example}

Since we will be considering base changing morphisms of modules, we need the notion of pulling back a sheaf of modules along a morphism of ringed spaces.
\begin{definition}\label{definition:pullback.module}
Suppose \( f \colon X \to Y \) is a morphism of ringed spaces and \( \E\) is a sheaf of modules over \( X \). The \textbf{pullback} \( f^* \E\) of \( \E \) along \( f \) is the following sheaf of \( \O_X \)-modules:
\[ f^* \E := \O_X \otimes_{f\inv \O_Y} f\inv \E  \]
\end{definition}
\begin{example}\label{example:pullback.module.vector.bundle}
Suppose \( Y \) is a smooth manifold and \(\pi \colon  V \to Y \) is a vector bundle. Let \( \F \) denote the sheaf of sections of \( V \). Given a smooth map \( f \colon X \to Y \), \( f^* \E \) is the sheaf of sections of the pullback bundle \( f^* V := E \times_{\pi, f} X \). 
\end{example}
The main use of the pullback operation is to define morphisms of sheaves of modules over different bases.
\begin{definition}\label{definition:sheaf.module.morphism}
Suppose \( \E \) and \( \W \) are sheaves of modules over \( X \) and \( Y \), respectively. A \textbf{morphism} of sheaves of modules consists of a pair:
\[ f\colon X \to Y \qquad F \colon  \E \to  f^* \W  \]
where \( f \) is a morphism of ringed spaces and \( F \) is a morphism of sheaves of \(  \O_X \)-modules. Such morphisms will be denoted \( (F, f) \colon \E \to \W \). 
\end{definition}
\begin{definition}\label{definition:sheaf.module.comorphism}
Suppose \( \E \) and \( \W \) are sheaves of modules over \( X \) and \( Y \), respectively.. A \textbf{comorphism} of sheaves of modules consists of a pair:
\[ f \colon X \to Y \qquad F \colon  \W \to f_* \E \]
where \( f \) is a morphism of ringed spaces and \( F \) is a morphism of sheaves of \( \O_Y \)-modules.

The category whose morphisms are morphisms of sheaves of modules over locally ringed spaces is denoted \( \Mod_{LRS} \).
\end{definition}
Notice that during this generalization to sheaves, we have swapped the meanings of comorphisms and morphisms. In other words, morphisms of sheaves of modules correspond to comorphisms of module pairs and comorphisms of sheaves of modules correspond to morphisms of module pairs. This is to ensure that our definition of morphism of sheaves of modules agrees with the usual notion for vector bundles. 
\section{Fiber determined modules}\label{section:fiber.determined.modules}
In general, an arbitrary module is a fairly wild object.
In this section we will discuss a special class of modules over locally ringed spaces that we call \emph{fiber determined}. Fiber determined modules have particularly nice topological properties while still being fairly general. Fiber determined modules do not seem to be often referenced in the literature. However, one treatment does appear in \textit{Smooth Manifolds and Observables} by Jet Nestruev~\cite{Nestruev2003} under the name `geometrical modules'.
\subsection{Definition}
To start, we will begin with the definition of a fiber of a module over a point. From a certain point of view, the fiber of a module is the `zero-th' order part of the stalk.
\begin{definition}\label{definition:fiberspace}
Suppose \( \E \) is a sheaf of modules over \( X \). Given a point \( x \in X \), let \( \O_{X, x} \) and \( \E_x \) denote the stalks at \( x \) and \( \m_x \subset \O_{X, x} \) the maximal ideal. The \textbf{fiber} of \( \E \) at \( x \) is defined to be:
\[ \Fib(\E)_x := \frac{\E_x}{\m_x \E_x}  \]
We write \( \Fib(\E) \) to denote the disjoint union of all fibers.

Given a section, \( \sigma \in \E(U) \), we write \( \Fib(\sigma)_x \in \Fib(\E)_x \) to denote the class of \( \sigma \) in the fiber at \( x \in M \).
\end{definition}
\begin{example}
Suppose \( \pi \colon E \to M \) is a vector bundle over a smooth manifold. Let \( \Gamma(E) \) be the associated sheaf of modules. The stalk \( \Gamma(E)_x \) is the module of germs of sections over \( x \in M \). The set \( \m_x \Gamma(E)_x \) is the module of germs of sections which \emph{vanish} at \( x \). Therefore, the quotient \( \Gamma(E)_x / m_x \Gamma(E)_x \) is canonically isomorphic to the fiber \( \pi^{-1}(x) \).
\end{example}
Essentially, a fiber determined module is a module \( \E \) such that \( \Fib(\E) \) completely characterizes \( \E \).
\begin{definition}\label{definition:fiber.determined}
Suppose \( \E  \) is a sheaf of modules over \( X \). We say that \( \E \) is \textbf{fiber determined (FD)} if for all \( U \subset X \) open and \( e \in \E(U) \) we have that \( e = 0 \) if and only if for all \( x \in U ,\   \Fib(e)_x = 0 \). 
\end{definition}
Not every module is fiber determined. We will see a few examples in a moment. However, every module has a fiber determined module underlying it. We call this process \emph{fiber determinization}.
\begin{lemma}\label{label:fiber.determinization}
Suppose \( \E  \) is a sheaf of modules over \( X \), then there exists a unique, up to isomorphism, module \( \E^{FD} \) together with a surjective module homomorphism \( \pi \colon \E \to \E^{FD}\) with the following property: For all fiber determined \( \W \) and \( \O_x \)-module morphisms \( F \colon \E \to \W \), there exists a unique morphism \( F' \colon \E^{FD} \to \W \) such that \( F = F' \circ \pi \).
\[
\begin{tikzcd}[column sep=large]
\E \arrow[dr, "F"] \arrow[d, "\pi"] & \\
\E^{FD} \arrow[r, dashed, "\exists! F'"] & \W
\end{tikzcd}
\]
\end{lemma}
\begin{proof}
The uniqueness follows from the fact that \( \E^{FD} \) satisfies a universal property. We only need to show existence which we will do by construction.

Let \( K \subset \E \) be the following submodule. For \( U \subset M \) open let:
\[ K(U) := \{ e \in \E(U) \ : \ \forall x \in U , \ \Fib(e)_x = 0 \} \]
Then \( \E^{FD}:= \E/K \) is clearly fiber determined. We only need to show that \( \E^{FD} \) satisfies the universal property.

Suppose \( F \colon \E \to \W \) is an \( \O_X \)-module homomorphism and \( \W \) is fiber determined. If \( e \in K(U) \) then it follows that \( F(e) \in \m_x \W \) for all \( x \in U \). Since \( \W \) is fiber determined it follows that \( F(e) = 0 \). Since \( K \subset \ker(F) \), there exists a unique \( \O_X \)-module homomorphism \( F' \colon \E^{FD} \to \W \) satisfying the desired property.
\end{proof}
\begin{remark}
Given a smooth function \( f \colon M \to N \) and a fiber determined sheaf of modules \( \E \) over \( N \). It is not completely clear whether the pullback module \( f^* \E \) is fiber  \( f^*\W \) is fiber determined. If \( f^* \W \) is not fiber determined, in general, there is a case to be made for defining a \emph{fiber determined} morphism of modules \( F \colon \W \to \E \) to be a module homomorphism:
\[ F \colon \W \to f^* \E^{FD} \]
Unless one can guarantee the existence of a splitting \( f^* \E^{FD} \to f^* \E \), using this notion of morphism may represent a non-trivial modification. Every ordinary morphism would certainly induce a fiber determined module. However, it may not be the case that every fiber determined morphism of modules arises in this manner.
\end{remark}
\subsection{Examples}
Now, let us see some examples and counter examples for the fiber determined condition.
\begin{example}\label{example:fd.submodule}
Suppose \( \E \) is fiber determined and \( \W \subset \E \) is a submodule. Then \( \W \) is also fiber determined. To see this, notice that the inclusion \( \W \to \E \) factors through \( \W^{FD} \). This implies that \( \W \to \W^{FD} \) is injective and therefore an isomorphism.
\end{example}
\begin{example}\label{example:fd.sheaf.of.sections}
Suppose \( \pi \colon E \to M \) is a vector bundle over a smooth manifold. Then the sheaf of sections \( \Gamma_E \) is fiber determined. This can be seen quite easily using the fact that \( \Gamma_E \) is a locally free module.

If we combine this with Example~\ref{example:fd.submodule} we see that any submodule of a vector bundle is also fiber determined.
\end{example}
\begin{example}\label{example:fd.skyscraper}
Let \( \E  \) be a sheaf of modules over \( \R \) defined as follows. \( \E(U) = 0 \) if \( 0 \notin U \) and otherwise:
\[ \E(U) := \left\{ \text{germs of functions } f \in \Cinf_\R(U) \text{ around } 0 : \forall n \ge 0 ,\  \frac{\difn{n} f}{{\difer} x^n}(0) = 0 \right\}  \]
This is a well-defined sheaf of \( \Cinf_\R \)-modules and it is not fiber determined. A smooth function \( f \) with the property that \( frac{d^n f}{dx^n}(0) = 0 \) for all \( n \ge 0 \) must also be an element of \( \m_0 \E \). Therefore, \( \Fib(\E)_x = 0 \) for all \( x \in \R \).
\end{example}
\begin{example}\label{example:fd.bad.quotient}
Let \( \rho \colon \R \to \R \) be a smooth function such that \( \rho(t) = 0 \) for \( t \le 0 \) and \( \rho(t) > 0 \) for \( t > 0 \). 
Let \( \E \subset \Cinf_{\R} \) be the submodule generated by \( \rho \). 
Since \( \E \) is a submodule of a vector bundle, it is fiber determined. 
However, the quotient module \( \Cinf_\R / \E \) is not fiber determined. This is due to a contradiction with Lemma~\ref{lemma:fd.topological.lrc} which will appear in the next subsection.
\end{example}
\begin{example}
This example was suggested to the author by Luca Vitagliano. Suppose \( \R \) is an analytic manifold and \( \O_\R \) is the sheaf of real analytic functions on the line. Consider the sheaf \( \Omega_K \) of K\"ahler forms:
\[  \Omega_K (U) := \langle \dif f  , \ f \in \O_\R(U) \ |\  \forall f,g \in \O_\R(U) , \ \difer (fg) = g \dif f + f \dif g , \ \difer(f+g) = \dif f + \dif g \rangle \] 
It turns out that the sheaf of modules \( \Omega_K \) is not fiber determined. To see this, consider the element: 
\[  \alpha = \dif e^x - e^x \!\dif x  \in \Omega_K(\R) \]  
Given a point \( p \in \R \) we can expand using Taylor theorem to get:
\[ \exists f \in \O_\R(\R) , \quad e^x = e^p + e^p (x-p) + (x-p)^2 f(x) \]
If we substitute this expression for \( e^x \) into the definition of \( \alpha \), a direct calculation shows that  \( \alpha \in \m_p \O_\R \). This shows that \( \alpha \) evaluates to zero on every fiber. It is more difficult to show that \( \alpha \neq 0 \). To the best of the author's knowledge this is a folkloric result but an explanation of this fact can be found in a MathOverflow post of David Speyer\cite{MOKahler}.
\end{example}
\subsection{Over topological ringed spaces}
If \( X \) is a locally ringed space such that \( \O_X \) is a subsheaf of \( \C^0_X(U) \), we say that \( X \) is \emph{topological}.
\begin{definition}
Suppose \( X \) is a topological ringed space and \( \E \) is a sheaf of modules over \( X \). Given an open subset \(U \subset X \), the \textbf{coefficient topology}\cite{villatoro2019integration} on \( \E(U) \) is defined as follows: A subset \(V \subset \E(U) \) is open if and only if for all module homomorphisms \( F \colon \O_X(U)^n \to \E(U) \) we have that the inverse image \( F^{-1}(V) \) is open.

The coefficient topology can be extended to the fiber space. We say a subset \( V \subset \Fib(E) \) is open if and only for all open sets \( U \subset X \) and functions:
\[ ev_U \colon U \times \E(U) \to \Fib(\E) \qquad (x,e) \mapsto \Fib(e)_x \]
we have that \( ev_U^{-1}(V) \) is open.
\end{definition}
This topology is particularly nice when the module is fiber determined.
\begin{lemma}
Suppose \( \E \) is a fiber determined sheaf of modules over a topological ringed space \(X \). 
\begin{itemize}
\item Then for all \( U \subset X \) open we have that \( \E(U) \) is a Hausdorff topological vector space.
\item The set of zeros in \( \Fib(\E) \) is closed and for each \( x \in M \) and the subspace topology on the fiber \( \Fib(\E)_x \) makes it into a Hausdorff topological vector space.
\end{itemize}
\end{lemma}
\begin{proof}
For the first part, both addition and scalar multiplication are clearly continuous in this topology. So we only need to show that \( 0 \) is closed. For this, it suffices to show that \( F^{-1}(0) \subset \C^0_X(U)^n \) is closed for an arbitrary \( F \colon \C^0_X(U)^n \to \E(U) \).

Suppose \( \{u\}_{i=1}^\infty \) is a convergent sequence of elements in the kernel of \( F \) and let \( u_{\infty} \in \Cinf_M(U)^n \) be the limit. For an \( x \in U \) the function:
\[ Ev_x \colon \Cinf_M(U)^n \to \Fib(\E)_x \]
is continuous so it follows that \(u_\infty \in \ker Ev_x \) for all \( x \in U \). This implies that \( \Fib(F(u_{\infty}))_x = 0 \) for all \( x \) and therefore \( u_{\infty} \in \ker F \). This completes the argument that \( 0 \in \E(U) \) is a closed point.

Let \( 0_X \subset \Fib(\E) \) denote the set of all zeros. To show that \( 0_X \) is closed, it suffices to show that for any open set \( U \subset X \), the inverse image of \( 0_x \) under \( ev_U \) is closed. However, \( ev_U^{-1}(0_X) \) is precisely the set of elements of \( \E(U) \) that evaluate to zero at every point of \( U \). Since \( \E \) is fiber determined, it follows that \( ev_U^{-1}(0_X) = U \times \{ 0 \} \) which is closed.

The topology makes it clear that for all \( x \in M \) the addition and scalar multiplication operations on \( \Fib(\F)_x \) are continuous. Furthermore, the projection map \( \pi \colon \Fib(\E) \to X \) is also continuous. Therefore the intersection \( \{ 0 \} = 0_X \cap \pi^{-1}(x) \) is closed.
\end{proof}
One corollary of this result is that the value of a section in a fiber determined module is determined by its value on a dense open set.
%
%
\begin{lemma}\label{lemma:fd.topological.lrc}
Suppose \( \E \) is a fiber determined module over a topological ringed space \(X \). Let \( U \) be an open set and suppose \( e \in \E(U) \) has the property that the restriction of \( e \) to a dense open subset of \( U \) is zero. Then \( e = 0 \).
\end{lemma}
\begin{proof}
Given an element \( a \in \E(U) \), there is a natural function \( \sigma_e \colon M \to \Fib(\E) \) defined by \( \sigma_e(x) = \Fib(e)_x \). Furthermore, the topology on \( \Fib(\E) \) makes \( \sigma_e \) continuous. Therefore, \( \sigma_e^{-1}(0_X) \) is a closed set. Since we have assumed \( \sigma_e^{-1}(0_X) \) contains an open dense subset of \( U \), we conclude that it is all of \( U \).
\end{proof}'
%
%
%
%
\subsection{Over manifolds}
When the base space is a manifold, we can say even more about fiber determined modules. In this subsection we will see that it makes sense to talk about smooth time-dependent sections of \( \Cinf_M\)-modules and that for fiber determined \( \Cinf_M\)-modules we can make sense of integration and differentiation.
%
%
%
\begin{definition}\label{definition:coefficient.diffeology}
Suppose \( \E \) is a sheaf of modules over a smooth manifold \( M \). If \( N \) is a manifold, a function \( \alpha \colon N \to \E(M) \) is said to be smooth if there exists the following:
\begin{itemize}
\item A locally finitely collection of smooth functions \( \{ c^i \colon N  \to \Cinf_M(M) \}_{i \in I } \)
\item A set of elements \( \{ e_i \}_{i \in I} \in \E(M) \)
\end{itemize}
such that 
\[ \forall x \in N \qquad  \alpha(x) = \sum_{i \in I} c^i(x) e_i \]
Such a choice of \( \{ c^i \}_{i \in I} \) and \( \{ e_i \}_{i \in I} \) is called a \textbf{parameterization} of \( \alpha \).
\end{definition}
From the point of view of sheaves, this definition is quite reasonable since smooth functions \( \alpha \colon N \to \Cinf_M(M) \) can be seen as arising from global sections of the sheaf \( \pi_2^* \A \) over \( N \times M \). 
In the setting of fiber determined modules, this correspondence is one-to-one.
The following properties are easy to show from this definition:
\begin{itemize}
\item Parameterizations are not necessarily unique.
\item If \( \alpha \colon N \to \E(U) \) and \( \beta \colon N \to \E(U) \) are smooth then \( \alpha + \beta \) is smooth.
\item If \( \alpha \colon N \to \E(U) \) is smooth and \( f \in C_M^\infty(U) \) then \( f \alpha \) is smooth.
\item If \( \alpha \colon N \to \E(U) \) is smooth and \( V \subset U \) is open then the restriction \( \alpha_V \colon N \to \E(V) \) is smooth.
\item If \( f \colon N_1 \to N_2 \) is smooth and \( \alpha \colon N_2 \to \E(U) \) is smooth then \( \alpha \circ f \colon N_1 \to \E(U) \) is smooth.
\end{itemize}
When the module is not fiber determined, it is not necessarily the case that this notion of smoothness is robust enough for performing calculus with sections. However, fiber determined modules have a particularly nice property:
%
%
%
\begin{lemma}\label{lemma:geometric.module.has.calculus}
Let \( \E\) be a fiber determined sheaf of modules on a manifold \( M \).
Suppose \( \alpha \colon [0,1] \to \E(U) \) is smooth with a parameterization \( \alpha = \sum_{i \in I} c^i e_i \) then the following expressions do not depend on the choice of parameterization:
\[ \frac{\dif \alpha}{\dif t} := \sum_{i \in I} \frac{\dif c^i}{\dif t} e_i \qquad \int_{0}^t \alpha(s) \dif s := \sum_{i \in I} \left(\int_{0}^t c^i(s) \dif s \right) e_i  \]
Furthermore, the fundamental theorem of calculus holds:
\[ \alpha(t) = \alpha(0) + \int_0^t  \frac{\dif \alpha}{\dif t}(s) \dif s \]
\end{lemma}
\begin{proof}
For the first part, we will only show that \( \frac{\dif \alpha}{\dif t} \) is well defined since the case of the integral follows nearly identical reasoning. We need to show that the right hand side of the definition of \( \frac{\dif \alpha}{\dif t} \) is independent of the choice of parameterization.

Suppose \( \alpha (t) = \sum_{j\in J} \til c^j f_j \) is another parameterization of \( \alpha \). Since we only need to check independence of parameterization locally, we can assume without loss of generality that \( I \) and \(  J \) are finite sets with order \( n \) and \( m \) respectively. For each \( \epsilon > 0 \) and \( t_0 \in (0,1) \) we have that:
\[ \frac{ \alpha(t_0 + \epsilon) - \alpha(t_0) }{\epsilon} = \sum_{i=1}^n \frac{c^i(t_0 + \epsilon) - c^i(t_0)}{\epsilon} e_i = \sum_{j=1}^m \frac{\til c^j(t_0 + \epsilon) - \til c^j(t_0)}{\epsilon} f_j \]
Let \( F \colon \Cinf(M)^{n+m} \to \E \) be the module homomorphism associated to the generators \( \{ e_1 , \ldots , e_n , f_1 , \ldots f_m \} \subset \E(M) \). Then we see that the tuple:
\[ \left(  \left ( \frac{c^i(t_0 + \epsilon) - c^i(t_0)}{\epsilon} \right)_{i=1}^n ,  \left( -\frac{\til c^j(t_0 + \epsilon) - \til c^j(t_0)}{\epsilon} \right)_{j=1}^m \right) \]
represents a smooth path, parameterized by \( \epsilon \), in \( \Cinf_M(M)^{n+m} \) within the kernel of \( F \). Recall the topology on \( \E \) which we defined in Lemma~\ref{lemma:fd.topological.lrc}. In that lemma, we saw that this topology makes \( F \) smooth and the zero section closed. Furthermore, it is clear that smooth paths in \( \E(U) \) are continuous relative to this topology. Therefore, follows that the limit as \( \epsilon \) goes to zero is in the kernel of \( F \).

To see that the fundamental theorem of calculus holds, just apply the usual fundamental theorem of calculus to any given parameterization.
\end{proof}
%
%
%
\section{Lie-Rinehart structures}\label{section:lie.rinehart.structures}
In this section, we will finally get to the main object of study. Lie-Rinehart structures on locally ringed spaces and, more specifically, manifolds.
%
%
%
\subsection{Sheaves of Lie-Rinehart pairs}
Given a sheaf of modules \( \E \) over \( X \), we write \( \Der(\O_X, E) \) to denote the sheaf of derivations on \( \O_X \) with values in \( \E \). We may write \( \Der(\O_X) \) as short hand for \( \Der(\O_X,\O_X)\). In the smooth setting, the sheaf \( \Der(\O_X) \) can be canonically identified with the sheaf of vector fields on \( X \). Since we are operating over a ground field \( \R \), \( \Der(\O_X) \) is also canonically a sheaf of Lie algebras.
%
%
%
\begin{definition}\label{definition:lie.rinehart.structure}
A \textbf{Lie-Rinehart(LR) structure} on a ringed space \( X \) consists of the following data:
\begin{itemize}
\item A sheaf of modules \( \A \) over \( X \),
\item For each \( U \subset X \), a Lie bracket on \( \A(U) \) which together make \( \A \) into a sheaf of Lie algebras.
\item A module homomorphism \( \rho_\A \colon \A \to \Der(\O_X) \) called the \textbf{anchor map}
\end{itemize}
We require this data satisfy the following compatibility condition:
\begin{enumerate}
\item \( \rho \) is a morphism of sheaves of Lie algebras\footnote{ When all modules involved are finitely generated and projective, this condition can be dropped.};
\item for all \( U \subset X \) open with \( a, b \in \A(U) \) and \( u \in \O_X(U) \) we have that:
\[ [ a, ub ] =   u [a, b] + \rho_a(u) b \]
\end{enumerate}
A \textbf{fiber determined Lie-Rinehart structure} is a Lie-Rinehart structure on \( X \) for which the sheaf of modules \( \A \) is fiber determined.
\end{definition}
A Lie-Rinehart structure on \(X \) can equivalently be defined to be a topological space equipped with a sheaf valued in \( \LR \). By forgetting the Lie algebra portion of the sheaf of LR pairs, one recovers a sheaf of algebras and therefore such a sheaf always defines a ringed space.
%
%
%
\begin{example}[Spectrum of an LR pair]\label{example:lie.rinehart.structure.spectrum}
Suppose \( (\A,R) \) is an LR pair. Take \( X = \Spec(R) \) to be the topological space of prime ideals of \( R \). There is a standard basis for the topology on \( \Spec(R) \) which is indexed by elements of \( R \):
\[ \forall f \in R \qquad D_r := \{ \mathfrak{p} \in \Spec(R) \ : \ f \notin \mathfrak{p} \}  \]
On this basis, the structure sheaf takes the form \( \O_X (D_r) := R_f \) where \( R_f \) is the localization of \( R \) at the ideal generated by \( f \). 

We can similarly define a sheaf of modules on \( X \) by taking localizations of \( \A \):
\[ \forall f \in R \qquad D_f \mapsto \A_f \]
Given \( f \in R \) and an element of \( \A_f \), one defines a derivation of \( R_f \): 
\[\frac{a}{r} \in \A_f , \ \frac{s}{t} \in R_f \qquad  \L_{(a/r)} (s/t) :=  \frac{s \L_a(t) - t \L_a(s) }{r s^2}  \]
This defines an anchor map for our sheaf of \( \O_X \)-modules. 

Finally, there is also a natural Lie bracket:
\[\forall a/r , b/s \in A_f \qquad  [a/r, b/s ]  =  \frac{ [a, b]}{r s} + \frac{-\L_b (r) a}{s t^2} -\frac{-\L_a (s) b}{t s^2} \]
\end{example}
We can generalize morphisms and comorphisms of LR pairs to comorphisms and morphisms of LR spaces in a natural way.
%
%
%
\begin{definition}\label{definition:lie.rinehart.structure.morphism}
Suppose \( \A \) and \( \B \) are LR structures over \( M \) and \( N \), respectively. A \textbf{LR morphism} 
\[ (F, f ) \colon  (\A, X) \to  (\B, Y)  \]
is a morphism of sheaves of modules such that for all \( U \subset X \) open we have that 
\[ (F_U,f_U^* ) \colon (\A(U), \O_X(U) ) \to ( f^*\B(U), f\inv\O_Y(U))\] 
is a comorphism of Lie-Rinehart pairs. The category of Lie-Rinehart spaces equipped with this notion of morphisms is denoted \( \LieRS \).
\end{definition}
This definition directly generalizes the traditional definition of a morphism of Lie algebroids. As an exercise, we invite the reader to check that the spectrum operation defines a functor \( \CLR  \to  \LieRS \).
There is also a notion of comorphism.
%
%
%
\begin{definition}\label{definition:lie.rinehart.structure.comorphism}
Suppose \( \A \) and \( \B\) are LR structures over \( X \) and \( Y \), respectively. An \textbf{LR comorphism} 
\[ (F, f ) \colon  (\A, X) \to  (\B, Y)  \]
is a comorphism of sheaves of modules such that for all \( U \subset X \) open we have that 
\[ (F_U,f_U^* ) \colon (\B(U), \O_Y(U) ) \to ( f_* \A(U), f_* O_X(U))\] 
is a comorphism of Lie-Rinehart pairs. The category of Lie-Rinehart spaces with comorphisms is denoted \( \CLieRS \).
\end{definition}
%
%
%
\subsection{Factorization of (co)morphisms}
We saw earlier that morphisms and comorphism of LR pairs always factor through morphisms which cover the identity. The same thing is true of LR spaces. In order to state these theorems we need to define the analogues of actions and base changes for the sheaf setting.
%
%
%
\begin{definition}\label{definition:lie.rinehart.structure.action}
Suppose \( \A  \) is a LR structure over \( X \) and \( f \colon Y \to X \) is a morphism of ringed spaces. An \textbf{action} of \( \A \) along \( f \) consists of a \( \O_X \)-module homomorphism \( \alpha \colon \A \to f_* \Der \O_Y \) which is also a Lie algebra homomorphism. 
\end{definition}
%
%
%
\begin{example}\label{example:lie.rinehart.structure.action.lie.algebra}
Suppose \( \g \) is a Lie algebra and \( Y \) is a smooth manifold. We can think of \( \g \) as a Lie-Rinehart space with a trivial underlying space \( \{ * \} \). There is a unique smooth function \( f \colon Y \to \{ * \} \) and a module homomorphism \( \g \to f_* \Der \O_Y \) is the same as a Lie algebra homomorphism \( \g \to \mathfrak{X}(Y) \). In this way, we see that an action of \( \g \) as a Lie-Rinehart space is the same as a classical Lie algebra action.
\end{example}
%
%
%
\begin{lemma}\label{lemma:lie.rinehart.structure.action.yields.structure}
Suppose \( \A \) is an LR structure on \( X \) and \( \alpha \colon \A \to f_* \Der \O_Y \) is an action along \( f \). Then \( f^* \A \) is canonically an LR structure on \( Y \) and \[ (1 \otimes_{\Id_A} , f ) \colon \A \to f^* \A  \]
is a comorphism of LR spaces.
\end{lemma}
\begin{proof}
This lemma follows immediately from Lemma 2.17 which implies this Lemma at the level of stalks.
\end{proof}
%
%
%
\begin{theorem}\label{theorem:lie.rinehart.structure.comorphism.factors}
Suppose \( (F, f) \colon \A \to \B  \) is a comorphism of LR spaces covering \( f \colon Y \to X \). Then there exists a unique comorphism 
\[  (\overline{F},\Id_Y ) \colon f^* \A \to  \B )  \] 
such that 
\[  (F, f) = (\overline{F} , \Id_Y ) \circ (1 \otimes \Id_{\A}  ,f)  \]
\[\begin{tikzcd}[column sep = large]
\B \arrow[rr, bend left, "\overline{F}"] \arrow[r, "{1 \otimes \Id_{\A}}"] &  f^* \B  \arrow[r, "{\overline{F}}" ] & \A \\
\end{tikzcd}\]
\end{theorem}
\begin{proof}
Given an open subset \( U  \subset N \) and \( V = f\inv(U) \). Theorem~\ref{theorem:lie.rinehart.pair.morphism.factors} tells us we have a canonical factorization at the level of sections:
\[\begin{tikzcd}
& (f^* \B( V)) , \Cinf_M(V) \arrow[dr , "{(\overline{F}_U, \Id_{\Cinf_M}(V))}"] & \\
( \B(U), \Cinf_N(U) ) \arrow[ur, "{(1 \otimes \Id_A , f_U))}"] \arrow[rr, "{(F_U, f_U)}"] & & ( \A(V)), \Cinf_M (V))
\end{tikzcd} \]
The uniqueness of this factorization implies that it is compatible with the restriction maps of the sheaves involved. Therefore, this factorization descends to a factorization at the level of sheaves.
\end{proof}
In the case where \( \A \) and \( \B \) are Lie algebroids, this factorization can take place entirely in the category of Lie algebroids. This is due to the fact that the construction only utilizes the standard pullback of a vector bundle which is perfectly well defined.

Of course, there is also a factorization theorem for LR morphisms. This result is more interesting since it is actually false in the category of Lie algebroids. To state it we should define the analogue of base change of an LR structure along a map.
%
%
%
\begin{definition}\label{definition:lie.rinehart.structure.base.change}
Suppose \( \B \) is an LR structure on \( Y \) and \( f \colon X \to Y \) is a morphism of ringed spaces. The \textbf{base change} of \( \B \) along \( f \) is defined to be the following fiber product of \( \Der (\O_X) \) and \( f^* \B \) as sheaves of \( \O_X \)-modules:
\[ 
\begin{tikzcd}
f^! \B \arrow[r] \arrow[d] & f^* \B \arrow[d] \\
\Der(\O_X) \arrow[r] & \Der(f^{-1}\O_Y, \O_X )
\end{tikzcd}
\]
\end{definition}
At the level of stalks, this base change operation coincides with the notion of base change we established for Lie-Rinehart pairs. The operation is well defined since all fiber products exist in the category of sheaves of modules. This is in contrast to the setting of Lie algebroids where this fiber product is only well defined in the presence of some transversality assumptions on the smooth function \( f \colon X \to Y \).
The anchor map on \( f^! \B \) comes from the left leg of the fiber product diagram. The Lie bracket structure is locally defined by the Lie bracket on the base change of a Lie-Rinehart pair we defined earlier in Definition~\ref{definition:lie.rinehart.pair.base.change}. The top part of the fiber product diagram also shows that the base change \( f^! \B \) comes with a canonical projection \( (\pi, f) \colon f^! \B \to \B \).
%
%
%
\begin{theorem}\label{theorem:lie.rinehart.structure.morphism.factors}
Suppose \( (F, f) \colon \A \to \B \) is a morphism of LR spaces covering \( f \colon X \to Y \). Then there exists a unique morphism of LR spaces \( (\underline{F}, \Id_X ) \colon \A \to f^!\B \) such that:
\[ (F, f) = (\pi, f) \circ (\underline{F}, \Id_X) \]
\[\begin{tikzcd}[column sep = large]
\A \arrow[rr, "{F}", bend left] \arrow[r, "{\underline{F}}"] & f^! \B  \arrow[r, "{\pi}"] & \B 
\end{tikzcd}\]
\end{theorem}
\begin{proof}
The proof of this theorem is essentially the same argument as in Theorem~\ref{theorem:lie.rinehart.structure.comorphism.factors}, except that we rely on Theorem~\ref{theorem:lie.rinehart.pair.comorphism.factors} for the factorization at the level of sections.
\end{proof}
%
%
%
\section{Adjoint integrability}\label{section:frobenius.integrability}
Now, let us finally turn to fiber determined LR structures over smooth manifolds. In this section we will define a special class of fiber determined LR structures which we call adjoint integrable. It is adjoint integrable LR structures which will admit a particularly nice homotopy theory.
%
%
%
\subsection{Fiber determined LR structures}
%
%
%
\begin{proposition}\label{proposition:FDLR}
Suppose \( \A \) is an LR structure over a smooth manifold \( M \). Then there is a unique way to make \( \A^{FD}  \) into an LR structure such that the canonical projection map \( \pi \colon \A \to \A^{FD} \) is an LR morphism.
\end{proposition}
\begin{proof}
The uniqueness part follows from the fact that \( \pi \) is a surjective morphism of modules. Since \( (\X_M)^{FD} \cong \X_M \), the anchor map \( \rho \colon \A \to \X_M \) clearly descends to an anchor map \( \rho^{FD} \colon \A^{FD} \to \X_M \).

The only thing to check is that the Lie bracket on \( \A \) descends to a well-defined Lie bracket on \( \A^{FD} \). To prove this, it suffices to show that \( \ker \pi \) is an ideal of \( \A \).

In order to prove this, we first want to show that if \( \kappa \in \ker \pi_U \) then \( \kappa \in \m_x^2 \A(U) \) for all \( x \in U \). To see this note that \( \kappa \in \m_x \A(U) \) so let us write \( \kappa = u a \) for \( a \in \A(U) \) and \( u \in \m_x \). We can arrange so that \( u \) vanishes to first order at \( x \). Since the property we wish to prove is local, without loss of generality we can assume that \( u \) vanishes transversely. Therefore, let \( U' = U \setminus u^{-1}(0) \). Since \( u \) vanishes transversely, we conclude that \( U' \) is a dense open set and \( u \) is invertible over \( U' \). Therefore:
\[ (\kappa|_{U'})/u = a \]
Consequently, \( a \in \ker \pi_{U'} \). However, since \( U' \) is dense and open, Lemma~\ref{lemma:fd.topological.lrc} implies that \( a \in \ker \pi_{U} \) and in particular, \( a \in \m_x \A(U) \).

Now we will use this fact to prove that the bracket descends. We claim that for all \( x \in U \) and \( a \in \A(U) \):
\[ [a, \m_x^2\A(U)] \subset \m_x \A(U) \]
To see this, suppose \( u_1, u_2 \in \m_x \) and \( b \in \A(U) \). Then:
\[ [a, u_1 u_2 b ]  = u_1 u_2 [a, b] + \L_a (u_1 u_2) b = u_1 u_2 [a, b] + \L_a(u_1) u_2 b + u_2 \L_a(u_1) b \]
Since this expression is an element of \( \m_x \A(U) \). The claim holds. Since every element of \( \ker \pi \) is an element of \( \m_x^2 \A(A) \) for all \( x \in U \) it follows that \( [a, \ker \pi_U] \subset \ker \pi_U \).
\end{proof}
%
%
%
\subsection{Integrability}
We saw in the previous section that it makes sense to integrate and differentiate time-dependent sections of fiber determined modules. Let us now turn our attention to a special class of fiber determined modules for which a certain differential equation has a solution.
%
%
%
\begin{definition}
Suppose \( \A  \) is a fiber determined Lie-Rinehart structure on \(M \). We say that \( \A \) is \textbf{adjoint integrable} if the following property holds: Suppose we are given the following data:
\begin{itemize}
\item An open set \( U \subset M \)
\item a smooth function \( \alpha \colon \BBR \to \A(U) \)
\item and an element \( b_0 \in \A(U) \)
\end{itemize}
such that the flow \( \phi^t_{\rho(\alpha)} \) of \( \rho(\alpha(t)) \) exits for all time. Then there exists a unique smooth function \( b \colon \BBR \to \A(U) \) such that:
\begin{equation}\label{eqn:adjoint.equation}
 b'(t) = [b(t), \alpha(t) ] \qquad b(0) = b_0
\end{equation}
We call Equation~(\ref{eqn:adjoint.equation}) the \textbf{adjoint flow equation} and \( b(t) \) is called the \textbf{adjoint flow of \( b_0 \) along \( \alpha \)}.
\end{definition}
%
%
%
\begin{example}
Consider the tangent LR structure \( \X_M \) on a manifold \( M \). Given a time-dependent vector field \( X \colon \BBR \to \X_M(M) \) and a vector field \( Y \in \X_M(M) \) then the solution to the adjoint equation is the pushforward of \( Y \) along the flow of \( X \).
\end{example}
In the next proposition, we see that solutions to the adjoint flow equation give rise to LR isomorphisms.
%
%
%
\begin{theorem}\label{theorem:frobenius.int.adjoint.flow}
Suppose \( \A \) is a adjoint integrable LR structure on a manifold \( M \). Let \( \alpha \colon \BBR \to \A(M) \) be smooth such that the flow of \( X := \rho(\alpha) \) exists for all time. Then there exists a unique family of LR isomorphisms \( \Phi_\alpha^t \colon \A \to \A \) covering the flow of \( X \), \( \phi^t_X \colon M \to M \) such that for each \( b \in \A(M) \) we have that \( \Phi^t_\alpha(b) \) is the adjoint flow of \( b\) along \( \alpha \). 
\end{theorem}
\begin{proof}
The uniqueness assumption ensures that \( \Phi^t_\alpha \colon \A(M) \to \A(M) \) is well defined. It is also quite straightforward to check that \( \Phi^t_\alpha \) is a \( \BBR \)-linear isomorphism. In fact, the inverse of \( \Phi^t_\alpha \) is the function \( \Psi^t_\alpha \colon \A(M) \to \A(M) \) for which \( \Psi^t_\alpha(b) \) is the adjoint flow of \( b \) along \( - \alpha \).

To complete the proof, we need to show that 
\[ (\Psi^t_\alpha , (\phi^t_X)^* ) \colon (\A(M) , \Cinf_M(M) ) \to (\A(M) , \Cinf_M(M) \]
is an isomorphism of LR pairs. This amounts to proving that the following three equations hold:
\begin{enumerate}[(a)]
\item Compatibility with the anchor: 
\[ \forall t \in \BBR, \ \forall  b \in \A(M) \qquad  \rho( \Phi^t_\alpha (b)) = (\phi_X^t)_* (\rho(b)) \]
\item Compatibility with the bracket: 
\[ \forall t \in \BBR , \ \forall  b , c \in \A(M) \qquad  \Phi^t_\alpha ([b, c]) = [\Phi^t_\alpha (b) , \Phi^t_\alpha (c)] \]
\item Compatibility with the module structure: 
\[ \forall t \in \BBR, \ \forall  b \in \A(M), \ \forall u \in \Cinf_M(M) \qquad  \Psi_\alpha^t ( u b ) = \phi^t_\alpha (u) \Psi^t_\alpha (b) \]
\end{enumerate}
Note that the last condition was written in terms of \( \Psi \). This is due to the fact that the equation is easier to prove when written this way. For all three equations, we will show that the failure of each of these equations to hold are all trivial solutions to adjoint flow equations.

(a) Consider the following function:
\[ \zeta(t) := \rho( \Phi_\alpha^t(b)) - (\phi_X^t)_* \rho(b)  \]
Clearly condition (a) holds if and only if \( \zeta(t) = 0 \) for all \( t \).
Now let us compute the derivative:
\begin{align*}
\zeta'(t) &=  \frac{\dif}{\dif t} \left( \rho (\Phi_\alpha^t(b)) - (\phi_X^t)_* \rho(b) \right) \\
&= \rho ( [ \Phi^t_\alpha (b), \alpha(t)] ) - [ (\phi^t_X)_* \rho(b)  , X] \\
&= [ \Phi^t_\alpha (b) - (\phi^t_X))_* \rho(b) , \alpha(t)] \\
&= [ \zeta(t), \alpha(t)]
\end{align*}
Since \( \zeta(0) = 0 \), we conclude that \( \zeta(t) \) is the adjoint flow of zero along \( \alpha \).

(b) Let us fix \( b, c \in \A(M) \) and consider:
\[   \eta(t) := \Phi_\alpha^t ([b, c]) - [ \Phi^t_\alpha  (b) , \Phi^t_\alpha (c) ] \]
If we can show \( \eta(t) = 0 \) for all \( t \) then we have proved (b). Let us take the derivative:
\begin{align*}
\eta'(t) &= \frac{\dif}{\dif t} \left( \Phi_\alpha^t ([b, c]) - [ \Phi^t_\alpha  (b) , \Phi^t_\alpha (c) ] \right) \\ 
&= [ \Phi^t_\alpha ([b, c]) , \alpha(t) ] -  [ [\Phi^t_\alpha  (b), \alpha(t)] , \Phi^t_\alpha (c) ] - [ \Phi^t_\alpha  (b) , [\Phi^t_\alpha (c), \alpha(t) ]] \\
&= [ \Phi^t_\alpha ([b, c]) , \alpha(t) ]-  [ [\Phi^t_\alpha(b) , \Phi^t_\alpha (c)], \alpha(t) ] \\
&= [ \eta(t), \alpha(t)]
\end{align*}
So \( \eta(t) \) is the adjoint flow of \( 0 \) along \( \alpha \).

(c) Now suppose \( b \in \A(M) \) and \( u \in \Cinf_M (M) \). 
Let:
\[ \delta(t) := \Psi^t_\alpha (  u  b ) - (\phi^t_X)^* (u) \Psi^t_\alpha (b) \]
If we compute the derivative of this, we get that:
\begin{align*} 
\delta'(t) &= \frac{\dif }{\dif t} \left( \Psi^t_\alpha (  u  b ) - (\phi^t_X)^* u \Psi^t_\alpha (b) \right) \\
&= [\Psi^t_\alpha (  u  b ), - \alpha(t) ] - (\phi^t_X)^* (\L_X u) \Psi^t_\alpha (b) -(\phi^t_X)^* (u) [\Psi^t_\alpha (b), -\alpha]\\
&= [\Psi^t_\alpha (  u  b ), - \alpha(t) ] - (\phi^t_X)^* (\L_X u) \Psi^t_\alpha (b) -[ (\phi^t_X)^* (u) \Psi^t_\alpha (b), -\alpha] + \L_X (\phi^t_X)^* (u) b \\
&= [\Psi^t_\alpha (  u  b ), - \alpha(t) ]  -[ (\phi^t_X)^* (u) \Psi^t_\alpha (b), -\alpha]\\
&= [ \delta(t) , -\alpha(t)] 
\end{align*}
So \( \delta \) is the adjoint flow of \( 0 \) along \( -\alpha \).
\end{proof}
Since morphisms of modules give rise to maps at the level of fibers, the adjoint flow of a time-dependent section \( \alpha \) also defines an isomorphism of vector spaces \( \Fib(\A)_p \to \Fib(\A)_{\phi^t_X(p)}\) for all \( p \in M \).
%
%
%
\begin{theorem}[Stefan-Sussman-Frobenius Theorem for Lie-Rinehart Structures]\label{thm:frob.int.foliation}
Suppose \( \A \) is an adjoint integrable Lie-Rinehart structure on \( M \). Consider the distribution \( \mathcal{D}_\A:= \{ \rho(a)_x \ | \ a \in \A(M) , \ x \in M  \} \). Then there is a partition of \( M \) into maximal connected integral submanifolds of \( \mathcal{D}_\A \).
\end{theorem}
\begin{proof}
Let \( \F = \rho(\A) \). It is clear that \( \F(M) \subset \X(M) \) is a Lie subalgebra which generates \( \D \). 

According to the Stefan-Sussman theorem\cite{Sussmann}\cite{Stefan}, it suffices to show that \( \D \) is preserved under the flows of vector fields in \( \F(M) \). Suppose \( X, Y  \in \F(M) \), then there exist preimages \( \alpha, \beta \in \A(M) \). By the adjoint integrability assumption, we have a smooth family of LR morphisms \( \Phi^t_\alpha \colon \A \to \A \) which cover the flow of \( X \). 

We also know that \( \rho \circ \Phi^t_\alpha (\beta) \) must be a unique solution to adjoint equation on \( \X(M) \) with initial condition \( Y \). However, we already know that the solution to the adjoint equation is just the pushforward by the flow of \( X \). Therefore, \( (\phi_X^t)_*Y = \rho (\Phi^t_\alpha \beta) \in \F(M) \). Since \( Y \) is arbitrary, this shows that \( (\phi^t_X)_* \F(M) \subset \F(M) \) and therefore \( \dif \phi^t_X (\D) \subset \D \).
\end{proof}
%
%
%
\subsection{Examples}
We will now prove some results which will hopefully convince the reader that there are many interesting adjoint integrable LR structures. We will begin by showing that finitely generated and fiber determined LR structures are automatically adjoint integrable.
%
%
%
\begin{theorem}\label{thm:fin.gen.frob.int}
Suppose \( \A \) is a fiber determined LR structure on \( M \) and \( \A \) is locally finitely generated. Then \( \A \) is adjoint integrable.
\end{theorem}
\begin{proof}
Suppose \( \A \) is adjoint integrable. Partitions of unity allow us to conclude that it suffices to show the existence of solutions to the adjoint flow equation in a neighborhood of each point. Suppose \( x \in M \) and \( \A/ \m_x \A \) has rank \( k \). Let \( U \) be an open neighborhood such that \( \A(U) = < e_1 , \ldots e_k > \).
Now suppose \( \alpha \colon \BBR \to \A(U) \) with \( X := \rho(\alpha) \) complete on \( U \). For each \( e_i \) choose functions \( \{c^m_i(t) \} \subset \Cinf_M(U) \) such that:
\[ [\alpha,e_i ] = \sum_{i=1}^k c_i^m e_m \]
We can use these coefficients to define a time-dependent derivation \( D^t \colon \Cinf_M(U)^k \to \Cinf_M(U)^k \), which is uniquely determined by the property:
\[  \forall  1 \le i \le k \qquad D^t(\overbrace{0 , \ldots , 0}^{i-1}, 1, 0, \ldots , 0) := (c_i^1(t), \ldots , c_i^k(t)) \]
Since the anchor map is a Lie algebra homomorphism the symbol of \( D_a \) is \( \rho(\alpha) \).

If the flow of \(\rho(\alpha) \) exists for all time, so the usual theory of derivations on vector bundles tells us that there exists a family of module isomorphisms \( \til\Phi^t \colon \Cinf_M(U)^k \to \Cinf_M(U)^k \) covering the flow of \( X\) such that:
\begin{equation}\label{eqn:adjointeqnlift}
 \forall u \in \Cinf_M(U)^k \qquad \frac{\dif}{\dif t} \til\Phi^t (u) =  - D^t ( \til\Phi^t(u)  )
 \end{equation}
Now consider the projection \( \pi \colon \Cinf_M(U)^k \to \A(M) \). Since for all \( u \in \Cinf_M(U)^k \) we have that \( D^t (u) = [ \pi(u), \alpha] \) it follows that \( \pi \circ \til \Phi^t(u) \colon \BBR \to \A(U) \) is the adjoint flow of \( \pi(u) \) along \( \alpha \).

This proves existence of solutions, now we need to show uniqueness. Suppose \( u \in \Cinf_M(U)^k \) and \( b_0 = \pi(u) \). Assume that \( b(t) \) is an arbitrary solution to the adjoint flow equation with initial condition \( b(0) = b_0 \). Since \( b(t) \) is smooth and \( \pi \) is surjective, by possibly shrinking \(U \) we can construct a lift \( \til b \colon \BBR \to \Cinf_M(U)^k \) such that \( \pi \circ \til b = b \).

Let us look at the difference \( \delta(t) := \til \Phi^t(u) - \til b(t) \). If we compute the derivative, we see that \( \pi( \delta'(t)) = 0 \). In other words, \( \delta'(t) \in \ker \pi \) for all \( t \). We also have that \( \delta(0) \in \ker \pi \). Since we have assumed that \( \A \) is complete, \( \ker \pi \) is closed and therefore \( \delta(t) \in \ker \pi \) for all \( t \). This proves that \( b(t) = \pi (\til \Phi^t(u)) \).
\end{proof}
%
%
%
\begin{example}\label{example:lie.alg.frob.int}
Suppose \( A \to M \) is a Lie algebroid. Then \( \Gamma_A  \) is fiber determined and locally finitely generated so it is therefore adjoint integrable. The solution to the adjoint equation for Lie algebroid already appears in~\cite{CrFeLie} where it is called the `flow of a section'.
\end{example}
%
%
%
\begin{example}\label{example:image.frob.int}
Suppose \( \A  \) and \( \B \) are adjoint integrable LR structures on \( M \) and \( F \colon \A \to \B \) is an LR morphism covering the identity. Then the image sheaf \( F(\A) \) is a adjoint integrable LR structure.

To see this, first note that \( F(\A) \) is fiber determined since it is a submodule of a fiber determined module. Furthermore, since the kernel of \( F \) will be an ideal in \( \A \), compatibility with the Lie bracket and anchor map ensures that \( \Im(\F) \) is a LR substructure.
\end{example}
Singular foliations are a particularly interesting class of examples. Since every submodule of the sheaf of vector fields is automatically fiber determined, checking adjoint integrability is a little simplified.
%
%
%
\begin{example}\label{example:fin.gen.foliation.frob.int}
Suppose \( \F   \) is a locally finitely generated submodule of \( \X_M \). Then \( \F \) is adjoint integrable.
\end{example}
%
%
%
\begin{example}\label{example:anchor.im.frob.int}
Suppose \( \A \) is a adjoint integrable LR structure on \( M \). Then the image sheaf of the anchor map \( \F := \rho(\A) \) is a adjoint integrable singular foliation.
\end{example}
We will conclude this section with a low-tech statement of the adjoint integrability condition for singular foliations. 
\begin{example}\label{example:frob.int.foliation.condition}
Suppose \( \F  \) is a submodule of the vector fields \( \X_M \) which is closed under the Lie bracket. Then \( \F \) constitutes a adjoint integrable LR structure if and only if for any locally finite collection of functions \( \{ c^i (t) \}_i \subset \Cinf_M(M) \), elements \( X_i \in \F(M) \) and \( Y \in \F(M) \), there exists a locally finite collection \( \{ u^j(t) \}_j \in \Cinf_M(M) \) and elements \( \{ Y_j \}_j \subset \F(M) \) such that: 
\[ (\phi^t_X)_* Y = \sum_j u^j(t) Y_j \quad \text{ where } \quad X(t) := \sum_i c^i(t) X_i \]
\end{example}

\section{Lie-Rinehart homotopy}\label{section:lie.rinehart.homotopy}
In this section we will develop the first steps of a theory of homotopy groups for Lie-Rinehart structures. 
Our main theorem will show that the notion of a homotopy group is well defined for any Lie-Rinehart structure on a ringed space, so long as the ground field is taken to be \( \BBR \). 
To that end, in this section \( ( X, \O_X ) \) and \( (Y, \O_Y) \) will denote a locally ringed space over \( \BBR \). \( \A \) and \( \B \) will denote LR structures on \( X \) and \( Y \), respectively.

\subsection{Homotopy}
In this section we will give a general definition of homotopy which is well-defined for any LR structure. However, since the boundary conditions are stated in terms of the fibers, it is likely that the definition is most relevant for fiber determined LR structures.
\begin{definition}\label{defn:LR.homotopy}
Suppose \( \A  \) and \( \B  \) are LR structures over \( X \) and \( Y \), respectively.
Let \(F_0\) and \(F_1 \colon \A \to \B\) be LR morphisms of LR structures with base maps  \(f_0\) and \(f_1 \colon X \to Y\) respectively. A \textbf{LR homotopy} from \(F_0\) to \(F_1\) is a morphism:
\[
(H,h) \colon \A \times \X_{[0,1]} \to \B 
\]
such that for \( t = 0,1 \), we have that \((F_t, f_0) = (H,h) \circ (J_t, j_t) \).
\[\begin{tikzcd}
\A \arrow[r, "J_t"] \arrow[rr, bend left, "{F_t}"]  & \A \times \X_{[0,1]} \arrow[r, "H"] & \B  
\end{tikzcd}\]
Here \( J_t \colon \A \to \A \times \X_{[0,1]}  \) is the canonical inclusion covering \( p \mapsto (t, p) \).

If \( S \subset X \), we say that \( H \) is an \(S\)-preserving homotopy if for all \( x \in S \) and \( t \in [0,1] \) we have that \( \Fib(H(0, \frac{\partial }{\partial t}))_{(x, t)} = 0 \) and \( h(x, t) \) does not depend on \( t \).
\end{definition}
Note that our definition of homotopy means that \( h \) is a classical homotopy from \( f_0 \) to \( f_1 \) where \(X \) and \( Y \) are thought of as topological spaces. 
In the case that \( X \) and \( Y \) are smooth manifolds, and the structure sheaves are the smooth functions, the condition that \( h(x, t) \) is independent of \( t \) is redundant. 
Now let us see a few examples that will give us some tools for constructing homotopies.
\begin{example}[Constant homotopy]\label{example:constant.homotopy}
Suppose \(F \colon \A \to \B \) is a LR morphism covering \( f \colon X \to Y \). Let \( h \colon X \times [0,1] \to Y \) be the constant family of smooth maps defined by \( h(p, t) = f(p) \). 
Then the constant homotopy is defined as follows:
\[ H \left( 0 , \frac{\partial}{\partial t} \right) = 0 \qquad H(b , 0 ) = \pi_1^* F(b) \in \pi_1^* f^* \B \cong h^* \B \]
\end{example}
\begin{example}[Reverse Homotopy]\label{example:reverse.homotopy}
Suppose \( H \colon \B \times \X_{[0,1]} \to \A \) is an LR homotopy. 
Let \( r \colon [0,1] \to [0,1] \) be  \( r(t) = 1-t \). 
Then one can construct an isomorphism:
\[ R := \pi_1 \times (\dif r \circ \pi_2) \colon \A \times \X_{[0,1]} \to \A \times \X_{[0,1]}  \]
If we do a composition 
\[ \overline{H} := H \circ R \colon \A \times \X_{[0,1]} \to \B \] 
we get an LR homotopy which we call the \emph{reverse} of \( H \).
When \( H \) is an \( S \)-preserving-homotopy the reverse is also an \( S \)-preserving homotopy.
\end{example}
The next three examples show that the LR homotopy relationship is transitive.
\begin{example}[Reparameterization]\label{example:reparameterization}
Suppose \( H = \colon \A \times TI \to \B \) is an LR homotopy and \( \tau \colon [0,1] \to [0,1] \) is an endpoint preserving diffeomorphism. Then we can define a new LR homotopy:
\[ H^\tau := H \circ (\pi_1 \times (\dif \tau \circ \pi_2)) \colon \A \times \X_{[0,1]} \to \B  \]
which is called the reparameterization of \( H \) by \( \tau \).
\end{example}
\begin{example}[Lazy homotopies]\label{example:lazy.homotopy}
An LR homotopy \(H  \colon \B \times \X ([0,1]) \to \A \) covering \( h \colon N \times [0,1] \to M \) is called lazy if \( H \) is a constant homotopy in an open neighborhood of \( N \times \{ 0 , 1 \} \).
For an arbitrary \( H \), if \( \tau \colon [0,1] \to [0,1] \) is a reparameterization which is constant in a neighborhood of the boundary then  \( H^\tau \) is lazy.
\end{example}
\begin{example}[Concatenation]\label{example:concatenation}
Suppose \( H \colon \B \times \X_{[0,1]} \to \A \) and \( H' \colon \B \times \X_{[1,2]} \to \A \) where \( H  \) is a homotopy from \( F_0 \) to \( F_1 \) and \( H' \) is a homotopy from \( F_1 \) to \( F_2 \). We can define an LR morphism \( H' \cdot H \colon \A \times \X_{[0,2]} \to \B \) by defining it on an open cover:
\[ H' \cdot H |_{X \times [0,1)} = H' \qquad H \cdot H |_{X \times (1,2]} = H'_{+1}  \]
\[ H' \cdot H |_{X \times (1-\epsilon , 1 + \epsilon)} = F_1 \circ \pi_1 \]
In the above, \( H'_{+1} \) is just the translation of \( H' \) so that it is defined on the domain \( X \times [1,2] \) and \( \epsilon \) is a positive real such that \( H \) and \(H' \) are constant in an \( \epsilon\)-neighborhood of the boundary.

Technically, this is not a homotopy due to the fact that it is defined on the domain \( [0,2] \). Therefore, we define the \emph{concatenation} of \( H \) and \( H' \) to be:
\[ H' \odot H  := (H' \cdot H) \circ (\pi_1 \times (\dif \rho \circ \pi_2)) \colon \A \times \SC \to \B \]
where \( \rho \colon [0,1] \to [0,2] \) is the function \( \rho(t) = 2t \).
\end{example}
\section{Homotopy groups and groupoids}\label{section:homotopy.groups}
The homotopy groups are defined to be equivalence classes of morphisms from \( n \)-dimensional cubes \( \X_{[0,1]^n} \). 
We use cubes due to the fact that they have a convenient parameterization and the sheaf of vector fields on an \( n \)-dimensional cube is fairly easy to work with. 
Points in \( [0,1]^n \) will typically be denoted by \( \hat t = (t^1, \ldots , t^n) \) and the corresponding generators of \( \X_{[0,1]^n} \) are written \( \frac{\partial}{\partial{t^i}} \) for \( i=1, \ldots , n \).
\begin{definition}\label{definition:ncube}
Suppose \(\A\) is an LR structure on \( X \). \textbf{A Lie-Rinehart \(n \)-cube} is an LR morphism 
\[
\begin{tikzcd}
\X_{[0,1]^n} \arrow[d, squiggly] \arrow[r, "G"] & \A \arrow[d, squiggly] \\
{[0,1]^n} \arrow[r, "\gamma"] & X
\end{tikzcd}
\]
The \textbf{base-point} of an LR cube is defined to be \( \gamma(0, \ldots , 0) \in X \).
\end{definition}
\begin{example}
A \textbf{Lie-Rinehart path} is a 1-d LR cube \( G \colon \X_{[0,1]} \to \A \). The base point of an LR path is the starting point of the path and if two LR paths are related by a \(\{0,1\}\)-preserving homotopy then the associated underlying paths in \(X \) are homotopic.
\end{example}
\begin{example}
An LR \(n\)-cube is a \textbf{Lie-Rinehart \(n\)-sphere} if for all \( i = 1, \ldots , n \) we have that \( \Fib(G(\frac{\partial }{\partial{t^i}}))_{\hat t} = 0 \) for \( \hat t \) in the set \(  \{ t^i = 0,1 \} \subset \partial [0,1]^n \) and \( \gamma|_{\partial [0,1]^n} = \gamma(0, \ldots , 0 ) \).
\end{example}
\begin{definition}
Let \( \A  \) be a LR structure on \(X \) and \( x \in X \) be a point. The \textbf{\(n\)th homotopy group} of \( \A \) at \( x \) is defined to be:
\[ \pi_n(\A, x) := \frac{ \{ \text{n-spheres with base-point }x \}}{\partial [0,1]^n\text{-preserving homotopy} }\]
The \textbf{fundamental groupoid} of \( \A \) is defined to be:
\[  \Pi(\A) := \frac{ \{\text{LR-paths} \}}{\partial [0,1]\text{-preserving homotopy}}\]
\end{definition}
\begin{example}
If \( M \) is a manifold and \( \A \) is the sheaf of sections of a Lie algebroid over \( M \), then \( \Pi(\A) \) is the so-called \emph{Weinstein groupoid} of the Lie algebroid. In particular, if the algebroid is integrable then \( \Pi(\A) \) is a smooth manifold and constitutes the canonical source simply-connected integration.
\end{example}
To define the product on these sets, we use the notion of reparameterization and the concatenation operation which we defined earlier: 

Suppose we are given two \(n \)-dimensional LR cubes \( G_1 \) and \( G_2 \). By viewing the last coordinate on \([0,1] \) as a deformation parameter, we can think of \( G_1 \) and \( G_2 \) as homotopies between \(n-1\)-dimensional LR cubes (with free boundary conditions). If \( G_1 \) is a homotopy from \( C_1\) to \( C_2 \) and \( G_2 \) is a homotopy from \(C_3 \to C_4 \) then we define:
\[ G_2 * G_1 := G_1^\tau \odot G_2^\tau \]
Where \( \tau  \colon [0,1] \) is a reparameterization of the interval which is constant in the neighborhood of the boundary.

\begin{theorem}\label{theorem:lie.rinehart.groups.well.defined}
Suppose \( \A \) is an LR structure over a locally ringed space \( X \) over \( \BBR \). Then the product operation defined above descends to a groupoid structure on \( \Pi(\A) \grpd X \) and a group structure on \( \pi_n(\A, x) \) for all natural numbers \( n \) and \( x \in X \).
\end{theorem}
We will dedicate the next subsection to the proof of this theorem. The idea is to construct the usual homotopy group(oids) using only smooth maps from cubes \( [0,1]^n \to M \) and translate this construction to the language of LR structures and morphisms of LR structures.

Before proceeding to the proof of Theorem~\ref{theorem:lie.rinehart.groups.well.defined}, let us show that it implies Theorems~\ref{theorem:weinstein.groupoid} and \ref{theorem:partition.into.leaves} from the introduction.
\begin{proof}[Proof of Theorem~\ref{theorem:weinstein.groupoid}]
Given an LR morphism \( F \colon \A \to \B \) covering \( f \colon X \to Y \), and an LR-homotopy \( H \colon \C \times \X_{[0,1]} \to \A \) then \( F \circ H \) is clearly also a LR homotopy. Therefore, we can conclude that the associations \( \A \mapsto \pi_n(\A, x) \) and \( \A \mapsto \Pi(\A) \) is a functorial assignment.

Furthermore, when \( \Gamma_A \) and \( \Gamma_B \) are the sheaves of sections of Lie algebroids then an LR homotopies \( \B \times \X_{[0,1]} \) are in direct correspondence with Lie algebroid homotopies \( B \times T[0,1] \to A \). Therefore, our definition of algebroid fundamental groups and groupoid are the same as the definition which appears in Brahic-Zhu\cite{Brahic_2011} and Crainic-Fernandes~\cite{CrFeLie}. 
\end{proof}
\begin{proof}[Proof of Theorem~\ref{theorem:partition.into.leaves}]
Compatibility with the anchor map for maps \( \X_{[0,1]} \to \A \) directly implies that the equivalence classes on \( X \) induced by the isomorphism classes of \( \Pi(\A) \) contained inside of the leaves of \( \rho(\A) \subset \X_M \). 

On the other hand, the leaves of \( \rho(\A) \) are generated by the time 1 flows of vector fields in \( \A \). Therefore, suppose \( V \in \rho(\A) \) and assume that the flow \( \phi_V^t \colon M \to M \) exists for \( t \in [0,1] \). Then for each \( x_0 \in M \) we can define an LR-path \( G \) which covers \( \gamma(t) := \phi_V^t(x_0) \):
\[ G\left(\frac{\partial}{\partial t}\right) := 1 \otimes V \]
Therefore, the leaves of \( \rho(\A) \) are precisely the same as the orbits of \( \Pi(\A) \grpd M \).
\end{proof}

\begin{remark} Interestingly, we did not need to assume that the base space \( X \) was a smooth manifold in order to define the fundamental groups and the fundamental groupoid. On the other hand, these homotopy groups are constructed by studying maps into \( \A \) where the domain is indeed a smooth manifold. In other words, the groups \( \pi_n(\A, x) \) can only capture information about \( X \) and \( \A \) which can be `seen' by a map coming from a unit cube (smooth structure included). 
\end{remark}

\subsection{Proof of Theorem~\ref{theorem:lie.rinehart.groups.well.defined}}
 Let us establish some terminology. In this subsection, \( X \) is a locally ringed space over \( \mathbb{R} \) and \( \A\) is an LR structure on \( X \). We will denote the standard coordinates on the \( n \)-cube by \( t^1 , \ldots t^n \). Furthermore, for \( i = 1, \ldots , n \) and \(j = 0,1 \) let 
\[ f^i_j \colon [0,1]^{n-1} \to [0,1]^n  \qquad f^j_n (t^1, \ldots , t^{n-1} ) := (t^1, \ldots , t^{j-1}, i , t^{j+1} , \ldots , t^n )  \]
be the face inclusion maps. If \( \gamma \colon [0,1]^n \to X \) is a morphism of ringed spaces, the \emph{start} of \( \gamma \) along the \( t^j\)-axis is defined to be \( \gamma \circ f^0_j \). Similarly the \emph{end} of \( \gamma \) along the \( t^j \)-axis is defined to be \( \gamma \circ f^1_j \). By using the differential, these definitions extend also to LR cubes. If \( C \colon \SCn{n} \to \A \) is an LR \( n\)-cube then the start and end of \( C \) along the \( t^j\)-axis are 
\[  C \circ \dif f_j^0 \colon \SCn{n-1} \to \A  \quad \mbox{ and } \quad  C\circ \dif f_j^1 \colon \SCn{n-1} \to \A \]
respectively. By taking the \( t^j \)-axis to be the deformation parameter, any LR \( n \)-cube can be thought of as a homotopy (with free boundary conditions) from the start to the end of \( C \) along the \( t^j \)-axis. Since the start and end are defined precisely on the boundary, we also note that the start and endpoint is invariant up to boundary preserving homotopy. For any natural number \( n \) let \( \C^n(\A) \) denote the set of all homotopy classes of \( n \)-cubes. Then for all \( 1 \le i \le n \) we have well-defined functions:
\[ S_i \colon \C^n(\A) \to \C^{n-1}(\A) \qquad E_i \colon \C^n(\A) \to \C^{n-1}  \]
where \( S_i \) is projection to the start and \( E_i \) is projection to the end.

An LR \( n \)-cube is said to be lazy on the \( t^j \)-axis if it is locally a constant homotopy in an open neighborhood of \( \{ t^j = 0,1 \} \). Recall in Example~\ref{example:concatenation} we saw that lazy homotopies have a well-defined concatenation operation. We will use \( \odot_j \) to denote concatenation along the \( t^j \)-axis. 

We can now explain Theorem~\ref{theorem:lie.rinehart.groups.well.defined} as a corollary of the following, more general proposition. For the sake of space, we will not perform all of the calculations in complete detail but provide a detailed sketch of the argument.
\begin{proposition}
Suppose \( \A \) is an LR structure on \( X \). Take \( 1 \le j \le n \) and let \( [G_1] , [G_2] \in C^n(\A) \) be homotopy classes of LR \( n \)-cubes represented by \( G_1 \) and \( G_2 \) such that \( E_i([G_1]) = S_i([G_2]) \). If \( G_1 \) and \( G_2 \) are lazy along the \( j \)-axis let:
\[ [G_2 ] * [G_1] := G_2 \odot_j G_1  \]
We claim that the above operation makes \( \C^n(\A) \) into a well defined set theoretic groupoid over \( C^{n-1}(\A) \).
\end{proposition}
\begin{proof}
(Well defined): Without loss of generality, we can prove only the case where \( i = n \). For this claim to hold, we first need to prove that every element in \( \C^n(\A) \) is represented by a lazy homotopy. Let \( \tau \colon [0,1] \to [0,1] \) be a boundary fixing, non-decreasing function such that \( \tau \) is locally constant in a neighborhood of the boundary. Given an LR \(n\)-cube \(C \) let:
\[ C^\tau := C \circ ( \pr_1 \times \dif \tau ) \colon \SCn{n-1} \times \SC \to \A  \]
Then a simple computation reveals that \( C^\tau \) is lazy. We also claim that there exists a boundary preserving homotopy from \( C \) to \( C^\tau \). To see this, choose a classical boundary preserving homotopy \( \til \tau \colon [0,1] \times [0,1] \to [0,1] \) which goes from the identity map to \( \tau \). Now take:
\[  H := G (\pi_1 \times \dif \til{\tau} (\pi_2 \times \pi_3) ) \colon   \SCn{{n-1}} \times \SC \times \SC \to A \]
This exhibits an explicit boundary preserving homotopy from \( C \) to \( C^\tau \). Therefore, every element of \( \C^n(\A) \) is represented by a lazy homotopy (along the \( t^n \)-axis). In fact, this argument can be used to prove a more general fact, that is: If \( \rho_i \colon [0,1] \to [0,1] \) for \( i = 1,2 \) are classically homotopic functions such that \( \rho_i(\{ 0, 1 \}) \subset \{ 0, 1 \} \), then it follows that there is a boundary preserving LR homotopy from \( C^{\rho_1} \) to \( C^{\rho_2} \). 
 
Now we need to show that the concatenation operation respects boundary preserving homotopies. Let \( C_1 \) and \( C_2 \) be LR \( n \)-cubes such that \( E_n( [C_1]) = S_n(C_2) \). Assume that there exist \( C_1' \) and \( C_2' \) such that \( [C_1'] = [C_1] \) and \( [C_2'] = [C_2] \). Then we claim that there exists a boundary preserving homotopy from \( G^\tau_1 \odot_{n} G^\tau_2 \) to \( (G'_1)^\tau \odot_n (G'_2)^\tau \). 

To see this, let 
\[ H_i \colon \SCn{n-1} \times \SC \times \SC \to \A  \] 
be boundary preserving homotopies from \( C_i \) to \( C_i' \) for \( i = 1,2 \). Let \( H_i^\tau \) be the reparameterization of \( H \) (along the \( t^n \)-axis). Then it follows that \( H_i^\tau \) is a boundary preserving homotopy from \( G^\tau_i \) to \( (G'_i)^\tau\). Furthermore, an easy calculation shows that \( H^\tau_2 \odot_n H^\tau_1 \) is a boundary preserving homotopy from  \( G^\tau_1 \odot_{n} G^\tau_2 \) to \( (G'_1)^\tau \odot_n (G'_2)^\tau \). This shows that concatenation respects boundary preserving homotopies.

(Units): Now, we will show that \( * \) has unit elements. We will show the existence of right-units but thee proof of left units is very similar. Suppose \( C \) is an LR \(n \)-cube with which starts at \( C_0 \) along the \( t^n \)-axis. Let \( 1_{C_0} \) denote the LR \( n \)-cube which represents the constant homotopy from \(C_0 \) to \( C_0 \). We claim that \( [C] * [u_{C_0}] = [C] \) and therefore \( [u_{C_0}] \) is a right unit. Let \( \tau_{2} \colon [0,1] \to [0,1] \) be defined as follows:
\[ \tau_{2} (t) = \begin{cases}\tau(2t ) & t \in [0,1/2] \\ 1 & t \in (1/2,1] \end{cases} \]
Then a direct calculation shows that \( C^\tau \odot_n \til C_0^\tau = C^{\tau_{2}} \). Since reparameterization preserves the homotopy class, the claim follows.

(Inverses): Suppose \( C \) is a lazy LR \( n \)-cube. Let \( r(t) := 1-t \) and:
\[ \overline{C} := C ( \pi_1 \times (\dif r \circ \pi_2)) \colon \SCn{n-1} \times \SC \to \A   \]
If the start of \( C \) is \( C_0 \) we claim that \( [\overline{C}] * [C] = [u_{C_0}] \) and so \( [\overline{C}] \) is a left-inverse of \( [C] \). Consider:
\[ \tau^{rev}(t) := \begin{cases} \tau(2t) & t \in [0,1/2]  \\  \tau(-2t) & t \in (1/2,1] \end{cases}  \]
Using the definition of concatenation, we get that \( \overline{C}^\tau \odot_n C^\tau = ( \overline{C} \odot_n C)^{\tau_{rev}} \). Since there is a boundary preserving homotopy from \( \tau_{rev} \) to the trivial map \(t \mapsto 0 \), we conclude that there exists a boundary preserving homotopy from \( \overline{C} \odot_n C)^{\tau_{rev}}\) to a constant homotopy. The argument for right-inverses is similar.

(Associativity): 
Let \( C_1,C_2 \) and \( C_3 \) be lazy LR \( n \)-cubes with start and end points such that \( ([C_1] * [C_2]) * [C_3] \) is well-defined. It follows that \( [C_1] * ([C_2] * [C_3] ) \) is also well-defined but we must show that they are equal. Consider the function:
\[ \rho (t) \colon \begin{cases} \frac{\tau(2t)}{4} &  t \in [0,1/2] \\ \frac{\tau(4t-2) + 1}{4} & t \in (1/2,3/4]\\  \frac{ \tau(4t-3)+1}{2}  & t \in (3/4 , 1] \end{cases} \]
To get a sense of this function, notice that it is a reparameterization of the unit interval such that:  
\[ \rho([0,1/2]) = [0,1/4] \quad \rho([1/2,3/4]) = [1/4, 1/2] \quad \rho([3/4, 1]) = [1/2,1] \] 
Furthermore, on each of this intervals \( \rho \) uses a rescaled version of our standard reparameterization \( \tau \). The point is that a direct calculation shows that:
\[ (C_1^\tau \odot C_2^\tau) \odot C_3^\tau = ((C_1 \odot C_2) \odot C_3 )^\rho \]
Since reparameterizations preserve homotopy classes we conclude that \( * \) is associative.
\end{proof}

\bibliographystyle{plain}
\bibliography{MEfolbib.bib}

\begin{thebibliography}{10}

\bibitem{AMsheaf}
Iakovos Androulidakis and Marco Zambon.
\newblock Stefan–sussmann singular foliations, singular subalgebroids and
  their associated sheaves.
\newblock {\em International Journal of Geometric Methods in Modern Physics},
  13(Supp. 1):1641001, 2016.

\bibitem{ardizzoni2020differentiation}
Alessandro Ardizzoni, Laiachi~El Kaoutit, and Paolo Saracco.
\newblock Towards differentiation and integration between hopf algebroids and
  lie algebroids, 2020.

\bibitem{Brahic_2011}
Olivier Brahic and Chenchang Zhu.
\newblock Lie algebroid fibrations.
\newblock {\em Advances in Mathematics}, 226(4):3105–3135, Mar 2011.

\bibitem{ChenLiu}
Zhuo Chen and Zhangju Liu.
\newblock On (co-)morphisms of lie pseudoalgebras and groupoids.
\newblock {\em Journal of Algebra}, 316, 01 2007.

\bibitem{CrFeLie}
Marius Crainic and Rui~Loja Fernandes.
\newblock Integrability of {L}ie brackets.
\newblock {\em Ann. of Math. (2)}, 157(2):575--620, 2003.

\bibitem{hartshorne}
Robin Hartshorne.
\newblock {\em Algebraic geometry}.
\newblock Springer-Verlag, New York-Heidelberg, 1977.
\newblock Graduate Texts in Mathematics, No. 52.

\bibitem{hm1993}
Philip~J. Higgins and Kirill C.~H. Mackenzie.
\newblock Duality for base-changing morphisms of vector bundles, modules, lie
  algebroids and poisson structures.
\newblock {\em Mathematical Proceedings of the Cambridge Philosophical
  Society}, 114(3):471–488, 1993.

\bibitem{MOKahler}
David E~Speyer (https://mathoverflow.net/users/297/david-e speyer).
\newblock Kahler differentials and ordinary differentials.
\newblock MathOverflow.
\newblock URL:https://mathoverflow.net/q/9723 (version: 2020-06-15).

\bibitem{Hueb}
Johannes Huebschmann.
\newblock Lie-rinehart algebras, descent, and quantization.
\newblock {\em Fields Inst. Commun.}, 43, 04 2003.

\bibitem{SylvainArticle}
Camille Laurent-Gengoux, Sylvain Lavau, and Thomas Strobl.
\newblock The universal lie $\infty$-algebroid of a singular foliation.
\newblock {\em arXiv:1806.00475}, 2018.

\bibitem{Nestruev2003}
Jet Nestruev.
\newblock {\em Smooth manifolds and observables}, volume 220 of {\em Graduate
  Texts in Mathematics}.
\newblock Springer-Verlag, New York, 2003.
\newblock Joint work of A. M. Astashov, A. B. Bocharov, S. V. Duzhin, A. B.
  Sossinsky, A. M. Vinogradov and M. M. Vinogradov, Translated from the 2000
  Russian edition by Sossinsky, I. S. Krasil'shchik and Duzhin.

\bibitem{LR_Original}
George~S. Rinehart.
\newblock Differential forms on general commutative algebras.
\newblock {\em Transactions of the American Mathematical Society},
  108(2):195--222, 1963.

\bibitem{Rinehart}
George~S. Rinehart.
\newblock Differential forms on general commutative algebras.
\newblock {\em Trans. Amer. Math. Soc.}, 108:195--222, 1963.

\bibitem{stacks-project}
The {Stacks project authors}.
\newblock The stacks project.
\newblock \url{https://stacks.math.columbia.edu}, 2020.

\bibitem{Stefan}
P.~Stefan.
\newblock Accessible sets, orbits, and foliations with singularities.
\newblock {\em Proc. London Math. Soc. (3)}, 29:699--713, 1974.

\bibitem{Sussmann}
H{{\'e}}ctor~J. Sussmann.
\newblock Orbits of families of vector fields and integrability of
  distributions.
\newblock {\em Trans. Amer. Math. Soc.}, 180:171--188, 1973.

\bibitem{Vaintrob}
A.~Yu. Vaintrob.
\newblock Lie algebroids and homological vector fields.
\newblock {\em Uspekhi Mat. Nauk}, 52(2(314)):161--162, 1997.

\bibitem{villatoro2019integration}
Joel Villatoro and Alfonso Garmendia.
\newblock Integration of singular foliations via paths, 2019.
\newblock arXiv:1912.02148.

\end{thebibliography}

\end{document}